\documentclass[12pt, a4paper]{article}

\usepackage{amsmath}
\usepackage{amsthm}
\usepackage{amssymb}

\usepackage[english]{babel}

\usepackage{graphicx}
\usepackage{tikz}

\usepackage{comment}
\usepackage{enumitem}
\usepackage{authblk}

\newcommand{\mbb}{\mathbb}

\theoremstyle{plain}
\newtheorem{thm}{Theorem}[section]
\newtheorem*{thm*}{Theorem}

\newtheorem{thmAlph}{Theorem}

\newtheorem{lem}[thm]{Lemma}
\newtheorem{lem*}{Lemma}
\newtheorem{cor}[thm]{Corollary}

\theoremstyle{remark}

\newtheorem{claim*}{Claim}
\newtheorem{rem}[thm]{Remark}
\newtheorem*{rem*}{Remark}

\newcommand{\N}{{\mathbb{N}}} % natural numbers {1, 2, ...}
\newcommand{\R}{{\mathbb{R}}} % reals
\newcommand{\NN}{{\mathbb{N}}} % natural numbers {1, 2, ...}
\newcommand{\RR}{{\mathbb{R}}} % reals
\newcommand{\ZZ}{{\mathbb{Z}}} % integers

\DeclareSymbolFont{bbold}{U}{bbold}{m}{n}
\DeclareSymbolFontAlphabet{\mathbbold}{bbold}

\newcommand{\calO}{{\mathcal{O}}}

\newcommand{\floor}[1]{\left\lfloor #1 \right\rfloor} % floor
\newcommand{\smallCalO}{{\scriptstyle \calO}}

\numberwithin{equation}{section}

\usepackage{float}
\usepackage{subfig}

\allowdisplaybreaks

\title{On the asymptotic behavior of Sudler products along subsequences}
\author{%Sigrid Grepstad \thanks{\ToDo{Thank Sigrid Grant}} and
Mario Neum\"uller \thanks{The author is funded 
by the Austrian Science Fund (FWF): Project F5509-N26, Project F5512-N26 (Special Research Program ``Quasi-Monte Carlo Methods: Theory and Applications'') and Project P29910-N35 (``Dynamics, Geometry, and Arithmetic of Numeration''). }}

\date{}

\begin{document}

	\maketitle

	\begin{abstract}
	Let $\alpha \in (0,1)$ and irrational. We investigate the asymptotic behaviour of sequences of certain trigonometric products (Sudler products) $(P_N(\alpha))_{N\in\NN}$ with $$P_N(\alpha) =\prod_{r=1}^N|2\sin(\pi r \alpha)|.$$ 
	
	More precisely, we are interested in the asymptotic behaviour of subsequences of the form $(P_{q_n(\alpha)}(\alpha))_{n\in\NN}$, where $q_n(\alpha)$ is the $n$th best approximation denominator of $\alpha$. Interesting upper and lower bounds for the growth of these subsequences are given, and  convergence results, obtained by Mestel and Verschueren (see \cite{VM16}) and Grepstad and Neum\"uller (see \cite{GN18}), are generalized to the case of irrationals with bounded continued fraction coefficients.
	\end{abstract}

	\centerline{
	\begin{minipage}[hc]{130mm}
		{\em Keywords:} Trigonometric product, continued fraction, Kronecker sequence, Ostrowski representation\\
		{\em MSC 2010:} 26D05, 41A60, 11J70 (primary), 11L15, 11K31 (secondary)
	\end{minipage}
	} 
	
	%%%%%%%%%%%%%%%%%%%%%%%%%%%%%%%%%%%%%%%%%%%%%%%%%%%%%%%%%%%%%%%%%%%%%%%%%%%%%%%%%%%%%%%%%%%%%%
	% INTRODUCTION
	%%%%%%%%%%%%%%%%%%%%%%%%%%%%%%%%%%%%%%%%%%%%%%%%%%%%%%%%%%%%%%%%%%%%%%%%%%%%%%%%%%%%%%%%%%%%%%

	\section{Introduction}
	
	In this article we are going to study the asymptotic behaviour of a certain trigonometric product given by
	\begin{equation}\label{eq:defPN}
		P_N(\alpha):=\prod_{r=1}^N|2\sin(\pi r \alpha)|
	\end{equation}
	for $N\in\N$ and some $\alpha\in\R$. Note that by \eqref{eq:defPN} it follows that $P_N(\alpha) =P_N(\{\alpha\})$. Further, for rational $\alpha$ the product $P_N(\alpha)$ will eventually become zero for sufficiently large $N$. Therefore we will only focus on irrational $\alpha$ in the interval $(0,1)$. 
	
	These products are often referred to as Sudler products and have first been studied more than 60 years ago by Erd\H{o}s and Szekeres \cite{Es59} in 1959  and by Sudler \cite{S64} in 1964. 
	Sudler products appear in a variety of different fields in pure and applied mathematics e.g. partition theory \cite{S64,W64}, Padé approximation, continued fractions (see \cite{L99} and the references therein), discrepancy and uniform distribution theory \cite{ALPET18,HP17} and in different topics related to mathematical physics \cite{AB,AHS13,BD,GOPY84,KT11,KPF95}. We refer to \cite{VM16} for a very detailed and well structured introduction to Sudler products.
	
	In 2016 Mestel and Verschueren \cite{VM16} studied a certain subsequence of $(P_N(\alpha))_N{\in\NN}$ for the special case of $\alpha$ being the golden ratio. The main result of \cite{VM16} reads as follows.
	\begin{thmAlph}[{\cite[Theorem~3.1]{VM16}}]\label{thm:MVMainResult}
		Let $\varphi=(\sqrt{5}-1)/2$ and $(F_n)_{n\geq0} = (0,1,1,2,3,\ldots)$ the Fibonacci sequence. Then there exists a constant $c>0$ such that
		\begin{equation}\label{eq:ConvGoldenRatio}
			\lim_{n\to\infty}P_{F_n}(\varphi)=c. 
			\end{equation}
	\end{thmAlph}
	It is well known that the Fibonacci sequence $(F_n)_{n\geq0}$ is the sequence of best approximation denominators of the golden ratio. Looking at Theorem~\ref{thm:MVMainResult} the rather natural question arises if one can generalize the convergence behaviour described in \eqref{eq:ConvGoldenRatio} to other irrationals $\alpha$ as well. This question was answered by Grepstad and Neum\"uller \cite{GN18} for quadratic irrationals.
	\begin{thmAlph}[{\cite[Theorem~1.2]{GN18}}]\label{thm:QuadIrrConv}
	Let $\alpha=[0;\overline{a_1,\ldots,a_{\ell}}]$ be a purely periodic quadratic irrational  \footnote{We write $[0;\overline{a_1,\ldots,a_{\ell}}]$ as an abbreviation for the periodic continued fraction expansion $[0;a_1,\ldots,a_{\ell},a_1,\ldots,a_{\ell},\ldots]$.}
	of period length $\ell$ with best approximation denominators $(q_n)_{n\in\N}$.
	Then there exist real constants $C_0,\ldots,C_{\ell-1}>0$ such that we have for each $k\in\{0,\ldots,\ell-1\}$ 
	\begin{equation*}
		\lim _{i\to \infty} P_{q_{\ell i+k}}(\alpha)=C_k.
	\end{equation*}
\end{thmAlph}
\begin{rem}\label{rem:genVersionQuadIrrConv}
	For simplicity Theorem~\ref{thm:QuadIrrConv} is stated only for the case of a purely periodic quadratic irrational but in fact very similar results also hold for arbitrary quadratic irrationals. For more details the author refers to \cite{GN18}.
\end{rem}

	The author would like to point out that the machinery behind Theorem~\ref{thm:MVMainResult} and Theorem~\ref{thm:QuadIrrConv} was one of the key ingredients for rather surprising results of Grepstad, Kaltenb\"ock and Neum\"uller \cite{GKN19} and Aistleitner, Technau and Zafeiropoulos \cite{ATZ} on the behaviour of $\liminf_{N\to\infty}P_N(\alpha)$ for special choices of quadratic irrationals $\alpha$. In order to put this into context it should be mentioned that almost 60 years ago Erd\H{o}s and Szekeres posed the following question in \cite{Es59}: What do we know about $\liminf_{N\to\infty} P_N(\alpha)$?
In the same paper they could prove that for almost all $\alpha$ it holds that $\liminf_{N\to\infty} P_N(\alpha)=0$ and they suggested that this result is true for all $\alpha$. Also Lubinsky made a similar suggestion in \cite{L99} after proving that $\liminf_{N\to\infty}P_N(\alpha)=0$ if $\alpha$ has unbounded continued fraction coefficients. It was rather surprising when Grepstad, Kaltenb\"ock and Neum\"uller \cite{GKN19} could prove that the golden ratio $\varphi$ is a counter example to these suggestions i.e.
\begin{equation}\label{eq:LiminfGreaterZero1}
	\liminf_{N\to\infty}P_N(\varphi)>0.
\end{equation}
For more detailed information on this topic we refer to the survey paper \cite{GrKaNe20}. Later Aistleitner, Technau and Zafeiropoulos \cite{ATZ} could fully characterize which quadratic irrationals $\alpha$ of period 1 (i.e. $\alpha$ with continued fraction expansions of the form $\alpha=[a;a,a,a,\ldots]$ for some $a\in\NN$) fulfill \eqref{eq:LiminfGreaterZero1}. 
	\begin{thmAlph}[{\cite[Theorem~6]{ATZ}}]\label{thm:LiminfGreaterZero2}
	Let $\alpha \in (0,1)$ and irrational with continued fraction expansion $\alpha=[a;a,a,a,\ldots]$. Then the following holds. 
	\begin{enumerate}
	\item If $a \in \{1,\ldots,5\}$ then 
		$$\liminf_{N\to\infty} P_N(\alpha)>0 \text{ and } \limsup_{N\to\infty} P_N(\alpha)/N <\infty.$$
	\item If $a\geq 6$ then
	$$\liminf_{N\to\infty} P_N(\alpha)=0 \text{ and } \limsup_{N\to\infty} P_N(\alpha)/N =\infty.$$
	\end{enumerate}
	\end{thmAlph}
	
	Inequality \eqref{eq:LiminfGreaterZero1} and Theorem~\ref{thm:LiminfGreaterZero2} show that the asymptotic behaviour of $P_N(\alpha)$ is not only depending on the structure of the continued fraction expansion of $\alpha$ but also on the actual size of the continued fraction coefficients.

	The main goal of this article is to understand the asymptotic behaviour of the sequence $(P_{q_n(\alpha)}(\alpha))_{n\in\NN}$ for a fixed irrational $\alpha$ and where $q_n(\alpha)$ is the $n$th best approximation denominator of $\alpha$. It will turn out that if $\alpha$ has bounded continued fraction coefficients the convergence behaviour described in Theorem~\ref{thm:QuadIrrConv} can be generalized to this case (in some sense). On the other hand if the continued fraction coefficients of $\alpha$ are unbounded this will lead to a different behaviour. We will give upper and lower bounds in this case.

Let us continue with some important notations and the main results of this article.

%%%%%%%%%%%%%%%%%%%%%%%  Notation and main results %%%%%%%%%%%%%%%%%%%%%%%%%%%%%%%%%%%%%%%

\subsection{Notations and formulation of the main results}\label{subsec:MainResults}
	
	Let $\alpha\in(0,1)$ and irrational. We denote by $p_n(\alpha)$ and $q_n(\alpha)$ the $n$th best
	approximation numerator and denominator, respectively, of $\alpha$. For some sequence $(a_n)_{n\in\NN}\in\N$ we write $\alpha=[0;a_1,a_2,\ldots]$ for the continued fraction expansion (c.f.e.) of $\alpha$. The integers $(a_n)_{n\in\NN}\in\N$ are called the continued fraction coefficients (c.f.c.) of $\alpha$.
%Further we write $\alpha=[0;p_1,\ldots,p_h,\overline{a_1,\ldots,a_{\ell}}]$ for a c.f.e. of the following periodic form
%$$\alpha = [0;p_1,\ldots,p_h,a_1,\ldots,a_{\ell},a_1,\ldots,a_{\ell},\ldots].$$

	It will turn out that the first component which determines the asymptotic behaviour of $(P_{q_n(\alpha)}(\alpha))_{n\in\N}$ are the sequences $(\alpha^+_n(\alpha))_{n\in\NN}$ and $(\alpha^-_n(\alpha))_{n\in \N}$ given by
\begin{align}\label{eq:defAlphPlusAlphMinus}
\alpha^+_n(\alpha):=[a_n;a_{n+1},a_{n+2},\ldots] \text{ and }
\alpha^-_n(\alpha):=[0;a_{n-1},a_{n-2},\ldots,a_1].
\end{align}
The sequences $(\alpha^+_n(\alpha))_{n\in\NN}$ and $(\alpha^-_n(\alpha))_{n\in\NN}$ describe sequences of c.f.e. which are depending on the c.f.c. of $\alpha$. More precisely, for $\alpha_n^+(\alpha)$ we consider the c.f.e. consisting of the c.f.c. $a_m$ of $\alpha$ with $m> n$ and $\floor{\alpha_n^+(\alpha)} = a_n$. For $\alpha_n^{-}(\alpha)$ we consider the c.f.e. consisting of the coefficients $a_1,\ldots,a_{n-1}$ of $\alpha$, but in reverse order.\\
The second component that influences the asymptotic behavior of the sequence $(P_{q_n(\alpha)}(\alpha))_{n\in\N}$ is of the following form. For $\beta \in \RR$ let 
\begin{equation}\label{eq:DefDt}
		D_t(\beta):=\sum_{s=1}^t\{\beta s\}-\frac{1}{2}.
\end{equation}
This function has been extensively studied by Hardy and Littlewood \cite{HL21,HL22} and Ostrowski \cite{O22}, but also by Sós \cite{S57} and Brown and Shiue \cite{BS95}. Further references can be found in \cite{S92}.

In order to ease notation in what follows we set
\begin{equation}\label{eq:cn}
	c_n(\alpha):=\frac{1}{\alpha_n^+(\alpha) + \alpha_n^-(\alpha)}.
\end{equation}
\noindent
Moreover, we will drop the dependence on $\alpha$ of the quantities $p_n(\alpha)$, $q_n(\alpha)$, $\alpha^+_n(\alpha)$, $\alpha^-_n(\alpha)$ and $c_n(\alpha)$ if it is clear from context.

 The subsequent theorem relates the asymptotic behaviour of $(P_{q_{n}}(\alpha))_{n\in\N}$ to the sequences $(\alpha^+_n)_{n\in\NN}$, $(\alpha^-_n)_{n\in\NN}$ and $(D_t(\alpha_n^-))_{t\in\NN}$ for $n\in\NN$. 
	\begin{thm}\label{thm:PqnBehaviour}
		Let $\alpha\in(0,1)$ be irrational and set $M_n=\lfloor(q_n-1)/2\rfloor$. We have that
		\begin{equation}
			P_{q_n}(\alpha) = \Theta\left(c_n\exp\Big( - 2 c_n\sum_{t=1}^{M_n-1}\frac{D_t(\alpha_n^-)}{t(t+1)} \Big)\right) \text{ for } n\to\infty,
		\end{equation}
		where the implied constants \footnote{For two functions $f,g:\RR\rightarrow\RR$ and $a\in\RR\cup\{\infty,-\infty\}$ we write $f(x)=\Theta(g(x))$ for $x\to a$ if there exist constants $c,C>0$ such that for $x\to a$ we have $c|g(x)| \leq |f(x)| \leq C|g(x)|$.} are independent of $n$, and $\alpha_n^-$, $D_t(\alpha_n^{-})$ and $c_n$ are defined in \eqref{eq:defAlphPlusAlphMinus}, \eqref{eq:DefDt} and \eqref{eq:cn}, respectively.
		\end{thm}
		%If we investigate the different behaviour of $c_n(\alpha)$ and $D_t(\alpha_n^-)$ in the presence of unbounded or bounded c.f.c.
		We will see later on that both of the sequences $(c_n)_{n\in\NN}$ and $(D_t(\alpha_n^-))_{t\in\NN}$ change their behaviour depending on whether the c.f.c. of $\alpha$ are unbounded or bounded. It is rather easy to see that $(c_n)_{n\in\NN}$ either tends to zero or is bounded from above and below by some constant greater zero. The situation is more complex for $(D_t(\alpha_n^-))_{t\in\NN}$. Here the behaviour of both cases can be described with the help of the Ostrowski expansion of $t$ in base $\alpha_n^-$ (see Section~\ref{sec:Prelim} for more details). The above mentioned observations lead to the following corollary.
		
		\begin{cor}\label{cor:PqnBehaviourTypeOfCFC}
			Let $\alpha \in (0,1)$ be irrational and $M_n=\floor{(q_n-1)/2}$. Denote for each $t\in\{1,\ldots,M_n-1\}$
			the digits of the Ostrowski expansion of $t$ in base $\alpha_n^-$ by $v_j(t)$ i.e. $t = \sum_{j=1}^{N(t)} v_j(t)q_j(\alpha_n^-)$. 
			\begin{enumerate}
			 \item If $\alpha$ has bounded c.f.c. we have
					\begin{equation}
						K_2\leq P_{q_n}(\alpha) \leq K_1,
					\end{equation}
					where $K_1,K_2 \geq 0$ and independent of $n$.
					%\begin{equation}
						%P_{q_n}(\alpha) = \Theta(1) \text{ for } n \to \infty,
					%\end{equation}
					%where the induced constants are independnt of $n$.
		\item If $\alpha$ has unbounded c.f.c. we have
			%\begin{equation}
				%  K_4 c_n\exp(Y_n)\leq P_{q_n}(\alpha) \leq K_3 c_n\exp(Y_n),
			%\end{equation}
			%where $K_3\geq K_4>0$ are constants independent of $n$ and
			\begin{equation*}
				P_{q_n}(\alpha) = \Theta\left(\frac{1}{a_n}\exp\left(\frac{1}{a_n}Y_n(\alpha)\right)\right) \text{ for } n\to\infty,
			\end{equation*}
			where the implied constants are independent of $n$ and  
			\begin{equation}
				Y_n(\alpha) = \sum_{t=1}^{M_n-1} \frac{1}{t(t+1)} \sum_{j=1}^{N(t)}(-1)^{j-1}v_j(t) 
					\left(1- v_j(t)c_j(\alpha_n^-)\right).
			\end{equation}
		\end{enumerate}
		\end{cor}
		The assertions of Corollary~\ref{cor:PqnBehaviourTypeOfCFC} give insight in the asymptotic behaviour of $(P_{q_n}(\alpha))_{n\in\NN}$ for $\alpha$ having bounded or unbounded c.f.c. It will turn out that under certain assumptions it is even possible to achieve convergence in the presence of bounded c.f.c.

The situation changes if we consider irrationals $\alpha$ with unbounded c.f.c. which will be illustrated more detailed for the special case $\alpha=\mathrm{e}$ (Eulers number). It is well known that $\mathrm{e}=[2;1,2,1,1,4,1,1,6,1,\ldots]$ or to put it differently we have $a_{3i+1}=1$, $a_{3i+2}=2(i+1)$ for $i\geq0$ and $a_{3i}=1$ for $i\geq 1$. Especially we see that the structure of the c.f.e. of $\mathrm{e}$ still has a periodic behaviour in some sense. We can summarize this with the notation
$$ \mathrm{e} = [2;\overline{1,2n,1}]_{n=1}^{\infty}.$$
This periodic structure is reminiscent of a quadratic irrational of period $\ell=3$ and also the corresponding plot of $P_{q_n}(\alpha)$ might seduce one to think that maybe similar properties hold as in the quadratic irrational case (see Figure~\ref{fig:PlotEuler}). 
But as a consequence of the work of Lubinsky~\cite{L99}, which deals with much more general problems concerning Sudler products, it follows that 
\begin{equation}\label{eq:LimitPointsEuler}
	\lim_{i\to\infty}P_{q_{3i+2}}(\mathrm{e})=0.
\end{equation}
\begin{rem}\label{rem:LubProperty}
	The author refers to the fact that for $\alpha$ with unbounded c.f.c. we have $\lim_{i\to\infty}P_{q_{n_i}}(\alpha)=0$ if $(a_{n_i})_{i\in\NN}$ is strictly increasing.  This is not explicitly stated in \cite{L99}, but this statement is a direct consequence from more general results of \cite{L99}. This circumstance is pointed out more detailed in the survey paper \cite{GrKaNe20}.
\end{rem}

\begin{figure}[H]
	\begin{center}
		\subfloat[{Values of $P_{q_n}(\mathrm{e})$ for $n=1,\ldots,19$.}]
			{\includegraphics[scale=0.525]{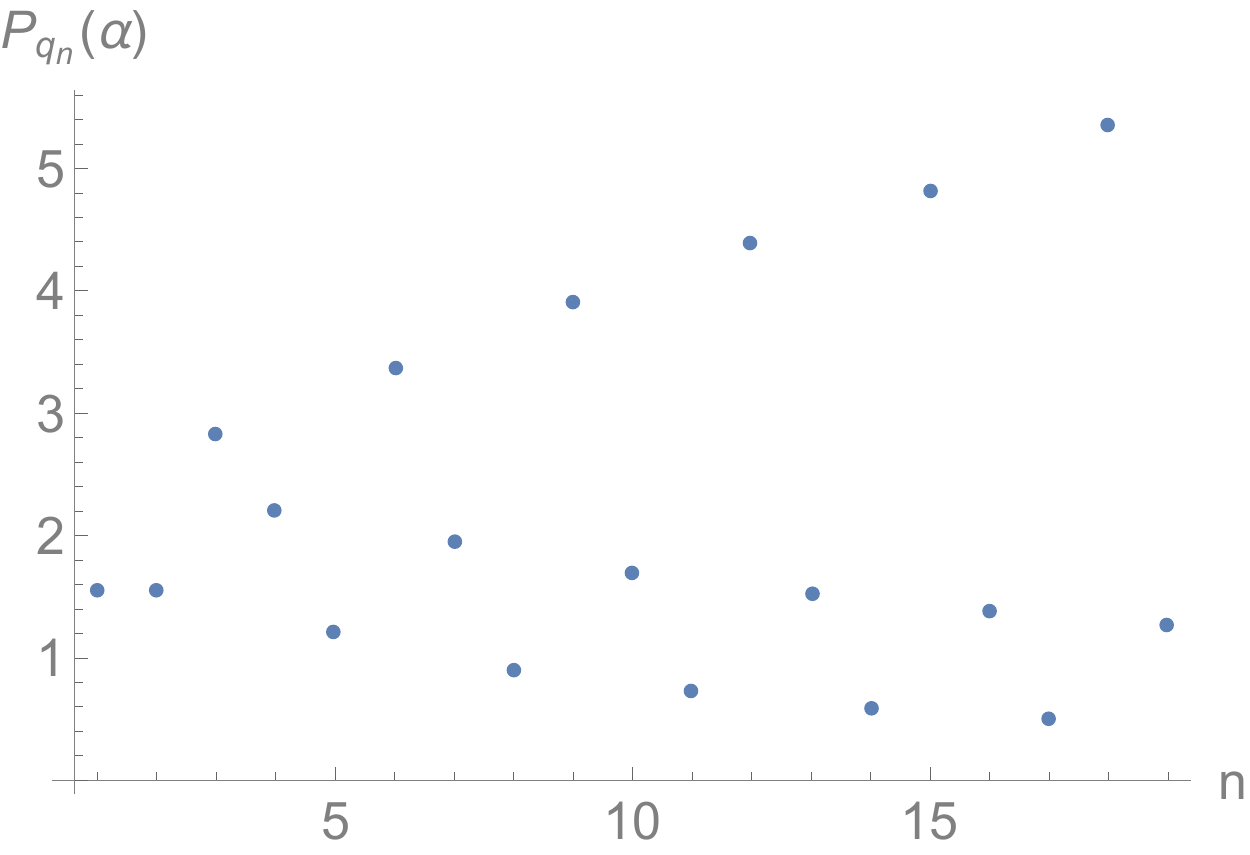}}
			\hspace{0.1cm}
		\subfloat[{Values of $P_{q_n}(\alpha)$ for $\alpha=[0;\overline{1,1,2}]$ and $n=1,\ldots,19$.}]
			{\includegraphics[scale=0.525]{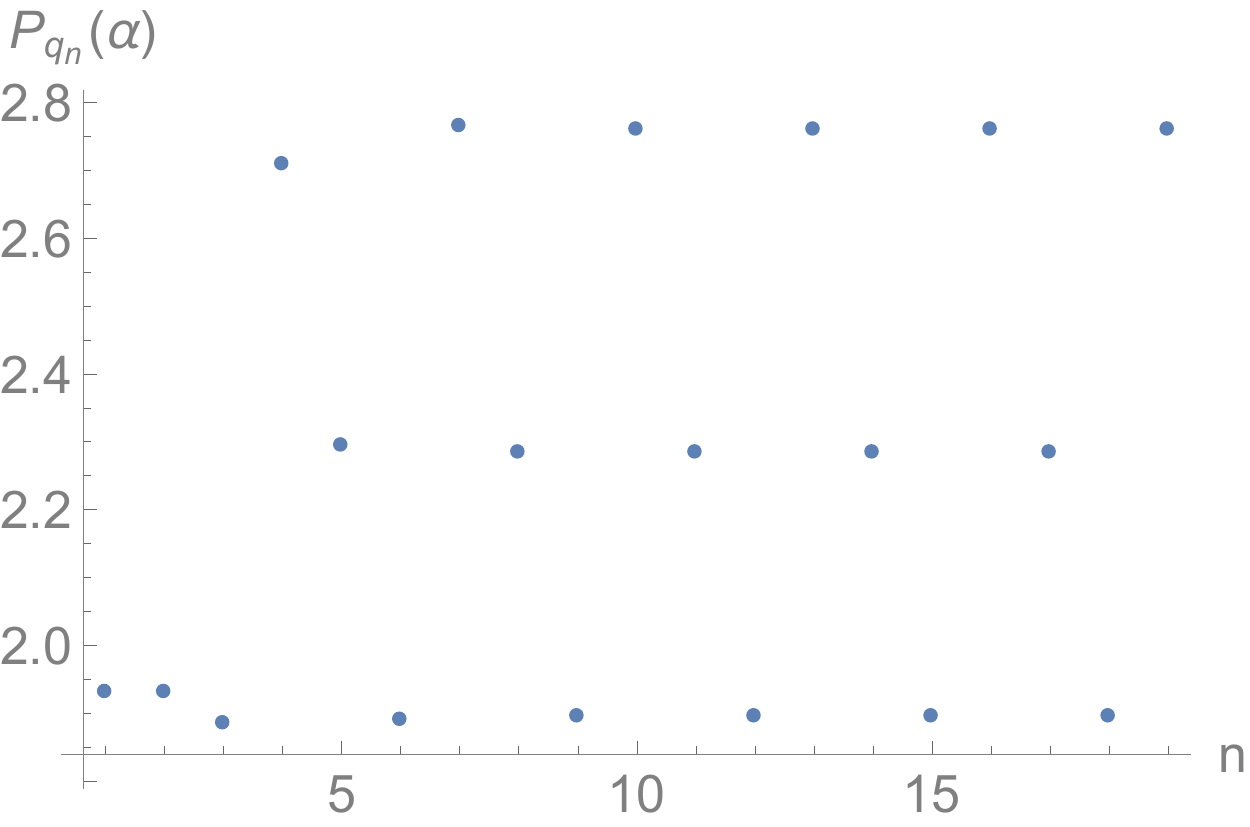}}
	\end{center}
	\caption{Comparision of $P_{q_n}(\alpha)$ for $\alpha=e$ and $\alpha$ being a quadratic irrational of period 3.}
	\label{fig:PlotEuler}
\end{figure}
The values for the limits in \eqref{eq:LimitPointsEuler} are already in contrast to the quadratic irrational case where all limit points of $P_{q_n}(\alpha)$ are strictly greater 0. The following corollary is a consequence of Theorem~\ref{thm:PqnBehaviour} and shows that the remaining subsequences for $k\in\{0,1\}$ either converge to zero or do not even converge at all. Moreover, the author would like to give a different proof for the already known limit for the case $k=2$ which is why this case is included in the statement of the corollary.
\begin{cor}\label{cor:behaviourEuler}
Let $e=[2;1,2,1,1,4,1,1,6,1,\ldots]=[2;\overline{1,2n,1}]_{n=1}^{\infty}$. Then
\begin{equation}
	\lim_{i\to\infty}P_{q_{3i+1}}(\mathrm{e})=0=\lim_{i\to\infty}P_{q_{3i+2}}(\mathrm{e}) \text{ and } \lim_{i\to\infty}P_{q_{3i}}(\mathrm{e})=\infty.
\end{equation}
\end{cor}
Another immediate consequence of Corollary~\ref{cor:PqnBehaviourTypeOfCFC} and the property stated in Remark~\ref{rem:LubProperty} is the following.
\begin{cor}\label{cor:ExistenceConvergentSubsequence}
	For every irrational $\alpha \in (0,1)$ the sequence $(P_{q_n}(\alpha))_{n\in\NN}$ has a convergent subsequence.
\end{cor}
In the presence of bounded c.f.c. it is even possible to give sufficient conditions for the convergence of a subsequence of $(P_{q_n}(\alpha))_{n\in\NN}$. The following theorem can be seen as a generalization of Theorem~\ref{thm:QuadIrrConv}.
\begin{thm}\label{thm:ConvBoundedCFC}
	Let $\alpha\in(0,1)$ be irrational with bounded c.f.c. $(a_n)_{n\in\N}$. Let $(n_i)_{i\in\N}$ be a strictly increasing sequence of positive integers. If additionally both of the limits
		\begin{equation}\label{eq:SuffCondConv}
				\begin{split}
				&\lim_{i\to\infty}  \alpha_{n_i}^+=\lim_{i\to\infty}[a_{n_i}; a_{n_i+1},a_{n_i+2},\ldots];\\
				&\lim_{i\to\infty} \alpha_{n_i}^-=\lim_{i\to\infty}[0;a_{n_i-1},a_{n_i-2},\ldots,a_1]
				\end{split}
		\end{equation}
		exist then there exists a constant $C>0$ such that
		\begin{equation}
			\lim_{i\to\infty}P_{q_{n_i}}(\alpha)=C.
		\end{equation}
\end{thm}
Observe that the convergence result for quadratic irrationals (Theorem~\ref{thm:QuadIrrConv}) can be fully recovered from Theorem~\ref{thm:ConvBoundedCFC} in the following way. Let $\alpha=[0;a_1,a_2\ldots]$ be a purely periodic quadratic irrational of period $\ell$ i.e. $a_n=a_{n\bmod{\ell}}$ (or $\alpha=[0;\overline{a_1,\ldots,a_\ell}]$). Then due to the periodic structure of the c.f.e. of $\alpha$ and with the choice $n^{(k)}_i= \ell i + k$ for $ k\in\{0,\ldots,\ell-1\}$ the following relations can be easily checked. (In what follows we write for simplicity $a_n$ instead of $a_{n\bmod{\ell}}$ and $a_{\ell}$ instead of $a_0$.)
\begin{align*}
	\lim_{i\to\infty}\alpha_{n^{(k)}_i}^-& =\lim_{i\to\infty} [0;a_{\ell i +k -1},a_{\ell i+k-2},\ldots,a_1]\\
	&=\lim_{i\to\infty}[0;a_{\ell i +k-1},\ldots,a_{\ell i+1},a_{\ell i},\ldots,a_1]\\
	&= \lim_{i \to \infty}[0;\underbrace{a_{k-1},\ldots,a_1}_{\text{preperiod}},\underbrace{a_{\ell},\ldots,a_1,\ldots ,a_{\ell},\ldots,a_1}_{i \text{ times the period } a_{\ell},\ldots,a_1}]\\
	&= [0;a_{k-1},\ldots,a_1,\overline{a_{\ell},\ldots,a_1}]
\end{align*}
and
\begin{align*}
	\lim_{i\to\infty}\alpha_{n^{(k)}_i}^+ &= \lim_{i\to\infty}[a_{\ell i +k}; a_{\ell i +k+1},a_{\ell i +k+2},\ldots] \\
	&=\lim_{i\to \infty} [a_k; \underbrace{a_{k+1},a_{k+2},\ldots,a_{\ell},a_1,\ldots,a_k}_{\text{period of length } \ell}, a_{k+1},a_{k+2},\ldots]\\
	&= [a_k;\overline{a_{k+1},\ldots a_{\ell},a_1,\ldots, a_{k-1},a_k}].
\end{align*}
Thus, Theorem~\ref{thm:ConvBoundedCFC} implies the convergence of $(P_{q_{\ell i +k}}(\alpha))_{i\in\NN}$ for each $k \in\{0,\ldots,\ell-1\}$ and $\alpha=[0;\overline{a_1\ldots,a_{\ell}}]$.

\begin{rem}
	The more general version of Theorem~\ref{thm:QuadIrrConv} described in Remark~\ref{rem:genVersionQuadIrrConv} can be recovered in a completely analogous way.
\end{rem}

Let us continue by having a more detailed look at the conditions \eqref{eq:SuffCondConv} stated in Theorem~\ref{thm:ConvBoundedCFC} and examine some examples for irrationals $\alpha$ with bounded c.f.c. satisfying these conditions.\\
Observe that a sequence $(\beta_i)_{i\in\NN}$ with $\beta_i=[b_0^{(i)};b^{(i)}_1,b_2^{(i)},\ldots]$ converges if and only if all the sequences which describe a single 
continued fraction coefficient (for $j\in\NN_0$ the $j$th coefficient of each $\beta_i$ is described by $(b_j^{(i)})_{i\in\NN}$) are constant from some point on. More precisely, the following are equivalent:
\begin{enumerate}
	\item $\lim\limits_{n\to\infty}[b_0^{(n)};b^{(n)}_1,b_2^{(n)},\ldots]$ exists;
	\item $\forall j \in \NN_0\,\exists N_j \in\NN: \lim\limits_{i\to\infty}b_j^{(i)}=N_j$ 
	%\item $\lim\limits_{n\to\infty} \min\limits_{k\in\NN_0}\{(b_1^{(n)},\ldots,b_k^{(n)})=(b_1^{(n+1)},\ldots,b_k^{(n+1)})\}$
\end{enumerate}
This means that the sufficient conditions of Theorem~\ref{thm:ConvBoundedCFC} are fulfilled if and only if for all $j\in\NN_0$ there exist positive integers $N_j^+$ and $N_j^-$ such that $\lim_{i\to\infty}a_{n_i+j}=N_j^+$ and $\lim_{i\to\infty}a_{n_i-j}=N_j^-$. Let us have a look at some simple examples fulfilling these properties:
\begin{enumerate}
	\item Probably the simplest example which is not a quadratic irrational is the irrational
	\begin{equation}\label{eq:SimpleAlpha}
		\alpha=[0;1,2,1,1,2,1,1,1,2,1,1,1,1,2,\ldots].
		\end{equation}
	If we choose for example the sequence $(n_i)_{i\in\NN}$ such that $n_1=2$ and $n_{i}=n_{i-1}+i+1$ for $i\geq2$ then $(n_i)_{i\in\NN}=(2,5,9,14,\ldots)$ reflects the positions of the twos in the c.f.e. of $\alpha$. By this choice we get that for all $j\in\NN$ $\lim_{i\to\infty}a_{n_i+j}=1$ and $\lim_{i\to\infty}a_{n_i}=2$. In other words
	\begin{align*}
		&\lim_{i\to\infty}\alpha_{n_i}^+ = [2;\overline{1}]=2+\varphi;\\
		&\lim_{i\to\infty}\alpha_{n_i}^-=[0;\overline{1}]=\varphi,
	\end{align*}
	where $\varphi=(\sqrt{5}-1)/2$ denotes the fractional part of the golden ratio. Therefore we get by Theorem~\ref{thm:ConvBoundedCFC} that $\lim_{i\to\infty} P_{q_{n_i}}(\alpha)$ exists and is strictly greater zero. (This is also visualized in Figure~\ref{fig:simpleAlpha}.)
	Observe that we still end up with convergence of $\alpha_{n_i}^+$ and $\alpha_{n_i}^-$ if we change the starting value $n_1$ to any other natural number. %This means that for $\alpha$ given in \eqref{eq:SimpleAlpha} the sequence $(P_{q_n}(\alpha))_{n\in\NN}$ has infinitely many limit points.

	\begin{figure}[H]
\begin{center}
		\subfloat[{$P_{q_n}(\alpha)$ for $\alpha=[0;1,2,1,1,2,\ldots]$ and\\ $1\leq n\leq 30$.}]
			{\includegraphics[scale=0.525]{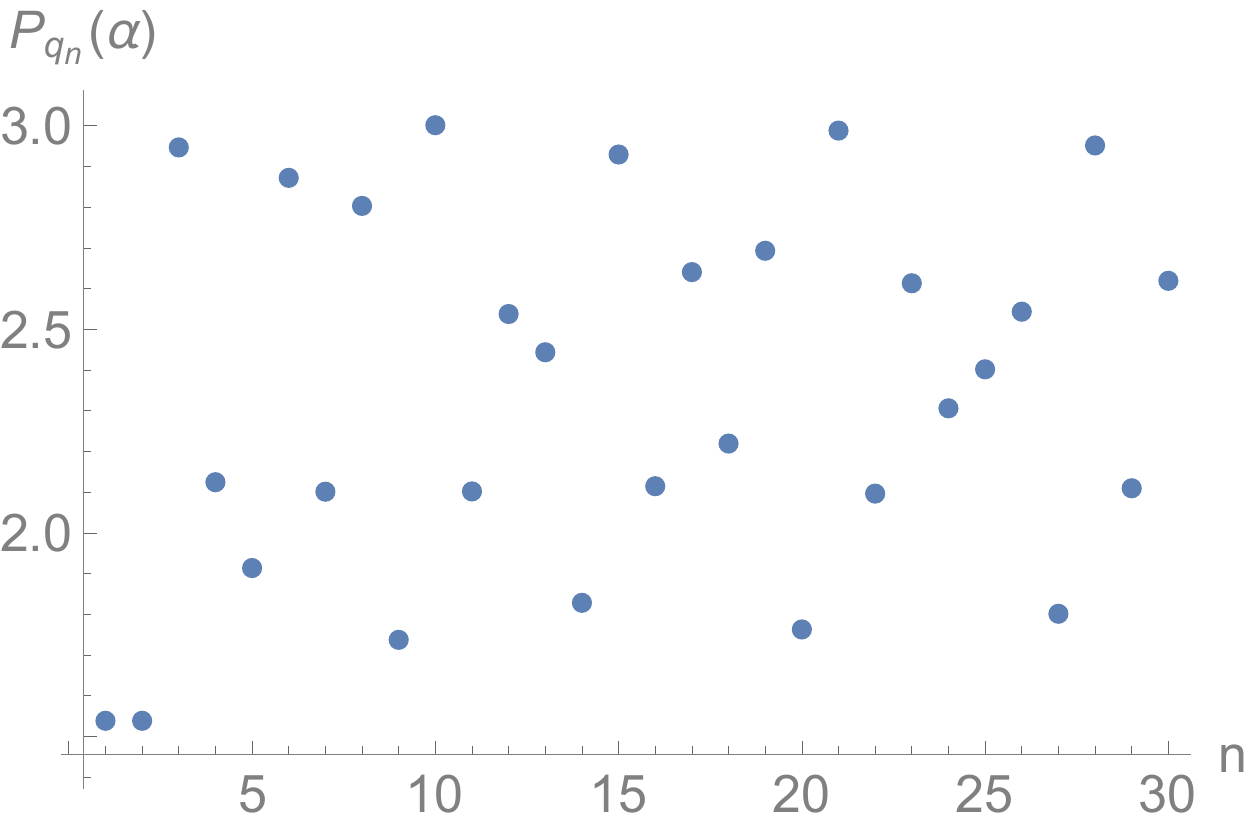}}
		\hspace{0.1cm}
		\subfloat[{$P_{q_{n_i}}(\alpha)$ for $\alpha$ given as in \eqref{eq:SimpleAlpha} and $n_i=n_{i-1} + i+1$, $n_1=1$ (dots), $n_1=2$ (triangles) and $n_1=3$ (stars).}]
			{\includegraphics[scale=0.525]{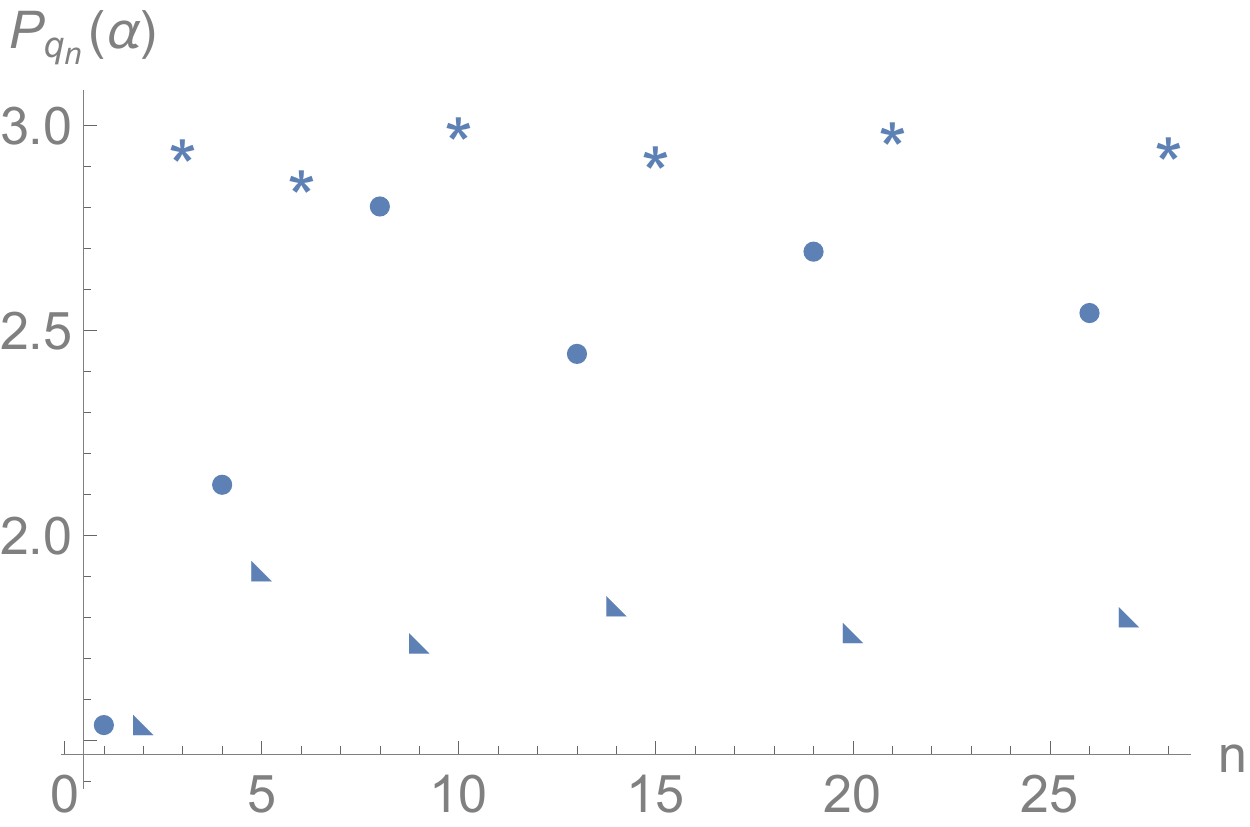}}
\end{center}
	\caption{Plots of $P_{q_n}(\alpha)$ for $\alpha$ given in \eqref{eq:SimpleAlpha}.}
	\label{fig:simpleAlpha}
\end{figure}
	
	\begin{comment}
	\item If $\alpha=[0;a_1,a_2,a_3,\ldots]$ with $A=\max_{n\in\N} \{a_n\}$ we identify $\alpha$ with the infinite word $(a_1a_2a_3\ldots)$ over the alphabet $U=\{1,\ldots,A\}^{\infty}$ and vice versa. For some $c\in\{1,\ldots,A\}$ let $(w_n)_{n\in\N}$ and $(\overline{w}_n)_{n\in\N}$ be sequences of words over the alphabet $U$ with $w_n=w_{n-1}\circ w_{n-2}$ and $\overline{w}_n=\overline{w}_{n-2}\circ \overline{w}_{n-1}$ respectively for $n\geq3$ and some starting values $w_1,w_2$ and $\overline{w}_1, \overline{w}_2$. Finally we consider the word
$$w=(w_1c\overline{w}_1,w_2c\overline{w}_2,w_3c\overline{w}_3,\ldots).$$
Let $(n_i)_{i\in\NN}$ describe the positions where the c.f.c. of $\alpha$ are equal to $c$ i.e. $n_1=|w_1|$ and for $i\geq 2$
$$n_{i+1}=
	\begin{cases}
		n_i+1 +|\overline{w}_i|+|w_{i+1}| & i \text{ even}, \\
		n_i+1 +|\overline{w}_{i+1}|+|w_{i}|& i \text{ odd}.
	\end{cases}$$ 
Then we get that \eqref{eq:SuffCondConv} is fulfilled and $\lim_{i\to\infty}P_{q_{n_i}}(\alpha)$ exists.
\end{comment}

\item Let $(t_n)_{n\geq0}$ be the Thue-Morse sequence on the alphabet $\{a,b\}\subset \NN$ i.e. $t_n=a$ ($t_n=b$) if the sum of binary digits of $n$ is even (odd respectively). Consider the irrational 
	\begin{align*}
		\beta&=[0;t_0,t_1,t_2,\ldots]\\
		&=[0;a,b,b,a,b,a,a,b,b,a,a,b,a,b,b,a,b,a,a,b,a,b,b,a,\ldots].
	\end{align*}
	Then for example the subsequence $(n_i)_{i\in\NN}$ with $n_i=2^{2i}$ fulfills the desired properties. In order to see this let us identify a continued fraction expansion with the word consisting of the corresponding continued fraction coefficients (e.g. $[0;a,b,a]$ is identified with the word $aba$ over the alphabet $\{a,b\}$).
	Now we can describe the c.f.e. of $\alpha$  with an infinite word over the alphabet $A=\{a,b\}$, which is recursively defined. For $n \in\NN$ let $w_{n} = w_{n-1}\circ \overline{w_{n-1}}$ and $w_0=a$. Here for $w \in A^{\infty}$ the word $\overline{w}$ is obtained by replacing each $a$ in $w$ with $b$ and vice versa. Now we have
	$$[0;t_1,\ldots,t_{2^i}] = w_{i} . $$
	Further we see that for $i \in\NN$
	 \begin{equation}\label{eq:InductionStep}
		[0;t_1,\ldots,t_{2n_{i}}] = w_{2i+1} = w_{2i}\overline{w_{2i}} = w_{2i-1}\circ\overline{w_{2i-1}}\circ\overline{w_{2i-1}}\circ w_{2i-1},
		\end{equation}
	where we used that for $w\in A^{\infty}$ we have $\overline{\overline{w}}=w$. Note that $|w_n| = 2^n$ and therefore $|w_{2i-1}\circ \overline{w_{2i-1}}| = 2^{2i} =n_i$. Now the definition of $w_n$ together with \eqref{eq:InductionStep} implies that for each $j\in\NN_0$ we have $\lim_{i\to\infty}t_{n_i+j}$ and $\lim_{i\to\infty}t_{n_i-j}$ exist. Thus we can apply Theorem~\ref{thm:ConvBoundedCFC} and it follows that $\lim_{i\to\infty}P_{q_{n_i}}(\beta)$ exist and is strictly greater zero.
\end{enumerate}

The rest of this article is structured as follows. We give a short summary of the most important results on continued fractions and Ostrowski representations in Section~\ref{sec:Prelim} and continue with a short description of the proof strategy for Theorem~\ref{thm:PqnBehaviour} and Theorem~\ref{thm:ConvBoundedCFC} in Section~\ref{sec:ProofStrategy}. The main part of this article is Section~\ref{sec:ProofLemmata}, which contains the essential steps (Lemma~\ref{lem:decomp} and Lemma~\ref{lem:StructurePqn}) for the proofs of Theorem~\ref{thm:PqnBehaviour} and Theorem~\ref{thm:ConvBoundedCFC}. Last but not least in Section~\ref{sec:ProofMainThm} we will state the actual proofs of Theorem~\ref{thm:PqnBehaviour} and Theorem~\ref{thm:ConvBoundedCFC}. The proofs of corollaries \ref{cor:PqnBehaviourTypeOfCFC} and \ref{cor:behaviourEuler} are stated in this section as well.

	%%%%%%%%%%%%%%%%%%%%%%%%%%%%%%%%%%%%%%%%%%%%%%%%%%%%%%%%%%%%%%%%%%%%%%%%%%%%%%%%%%%%%%%%%%%%%%
	% PRELIMINARIES
	%%%%%%%%%%%%%%%%%%%%%%%%%%%%%%%%%%%%%%%%%%%%%%%%%%%%%%%%%%%%%%%%%%%%%%%%%%%%%%%%%%%%%%%%%%%%%%

	\section{Preliminaries}~\label{sec:Prelim}
	
	Throughout the article we make use of the following common notations. For some functions $f,g:\RR\rightarrow\RR$  and $a\in\RR\cup\{\infty,-\infty\}$ we write
	\begin{enumerate}
		\item $f(x)=\calO(g(x))$ for $x\to a$ if there exists a constant $C$ independent of $x$ such that if $x$ is sufficiently close to $a$ then $|f(x)|\leq C|g(x)|$.
		
		\item $f(x)=\Theta(g(x))$ for $x \to a$ if there exist constants $c,C>0$ independent of $x$ such that if $x$ is sufficiently close to $a$ then $c|g(x)| \leq |f(x)|\leq C|g(x)|$.
		
		\item $f(x)=\smallCalO(g(x))$ for $x \to a$ if $\lim_{x\to a}|f(x)/g(x)| = 0$.
		
		%\item $f(x) \sim g(x)$ for $x\to\infty$ if $\lim_{x\to\infty} f(x)/g(x) = 1$. 
	\end{enumerate}
	
	Further we state a helpful Lemma which will be used at several points later on.
	
	\begin{lem}[{\cite[Lemma 4.3]{VM16}}]\label{lem:ProductSumLemma}
	For $n\geq 2$ and real numbers $a_t$, $t =1,2, \ldots , n$, satisfying $|a_t|<1$ and $A:=\sum_{t=1}^{n}|a_t|<1$, we have
	$$1- A<\prod_{t=1}^n(1+a_t) <\frac{1}{1-A}.$$
\end{lem}
	
	Let us continue by summarizing some basic knowledge and notations concerning the theory of continued fractions.
	
	\subsection{Continued fractions}
	The following notations and statements will be useful throughout the rest of this paper. In general, for any irrational, real $\alpha \in (0,1)$ whose continued fraction expansion is given by 
$$\alpha=[0; a_1, a_2, \ldots ],$$ 
we denote its $n$th convergent by $p_n/q_n$. The numerators $p_n$ and denominators $q_n$ are given recursively by 
\begin{equation*}
\begin{aligned}
q_0=0, \quad &q_1=1  \quad &q_{n+1}=a_{n}q_n+q_{n-1} ;\\
p_0=1, \quad &p_1=0  \quad &p_{n+1}=a_{n}p_n+p_{n-1} .
\end{aligned}
\end{equation*}
Note that the indexing of $p_n$ and $q_n$ is offset by one compared to what is normally seen in the literature. As a consequence, the $n$th convergent $p_n/q_n$ is smaller than $\alpha$ for every odd value of $n$, and greater than $\alpha$ for every even value of $n$. It follows readily from the recurrences above that 
\begin{equation}
\label{eq:pq}
p_nq_{n+1}-p_{n+1}q_n = (-1)^n,
\end{equation}
and as a consequence of this identity we have 
\begin{equation}
	p_nq_{n-1} = (-1)^n \pmod{q_n}.
\end{equation}
Another important consequence of \eqref{eq:pq} is the error bound
\begin{equation}
\label{eq:stderror}
\frac{1}{q_n(q_{n+1}+q_n)}\leq \left| \alpha - \frac{p_n}{q_n} \right| < \frac{1}{q_{n+1}q_n}
\end{equation}
for the $n$th convergent of $\alpha$.

Let us continue with stating two useful relations involving $\alpha_n^+$ and $\alpha_n^-$ defined as in \eqref{eq:defAlphPlusAlphMinus}. For $\alpha=[0;a_1,a_2,\ldots]$ it can be easily checked that
\begin{equation}\label{eq:AlphMinusFrac}
	\alpha_n^- = [0;a_{n-1},a_{n-2},\ldots,a_1] = \frac{q_{n-1}(\alpha)}{q_{n}(\alpha)}
\end{equation}
and it is well known that
\begin{equation}\label{eq:AlphaRewritten}
	\alpha =  \frac{p_n(\alpha) \alpha^+_n+p_{n-1}(\alpha)}{q_n(\alpha)\alpha_n^+ + q_{n-1}(\alpha)}.
\end{equation}

Further, we introduce a quantity which will appear frequently in the rest of this article
\begin{equation}\label{eq:Lambdan}
	\Lambda_n(\alpha) := q_n \alpha - p_n.
\end{equation}
Let us therefore summarize some well known properties of this object in the following Lemma.
\begin{lem}\label{lem:PropLambda}
	Let $\alpha \in (0,1)$ with best approximation denominators $(q_n)_{n\in\NN_0}$ and for $n\in\NN$ let $\Lambda_n(\alpha)$ be as in \eqref{eq:Lambdan}. Then it follows that
	\begin{enumerate}[label=(\alph*)]
		\item $|\Lambda_n(\alpha)|=(-1)^{n-1}\Lambda_n(\alpha)$ and $\Lambda_n(\alpha) \to 0$ for $n\to \infty$ . \label{it:LambdaFirst}
		\item $\dfrac{1}{2q_{n+1}}\leq |\Lambda_n(\alpha)| \leq \dfrac{1}{q_{n+1}}<\dfrac{1}{a_nq_n}$ \label{it:LambdaSec}
		\item $|\Lambda_n(\alpha)|q_n = c_n$, where $c_n$ is defined in \eqref{eq:cn}. \label{it:LambdaThird}
	\end{enumerate}
\end{lem}
\begin{proof}
The first part of \ref{it:LambdaFirst} follows from the fact that the $n$th convergent $p_n/q_n$ is smaller than $\alpha$ for every odd value of $n$, and greater than $\alpha$ for every even value of $n$. The second part of \ref{it:LambdaFirst} and the statement of \ref{it:LambdaSec} are direct consequences of \eqref{eq:stderror}. 
For \ref{it:LambdaThird} it follows by \eqref{eq:AlphaRewritten} that
$$q_n|\Lambda_n(\alpha)|=q_{n}^2\left|\alpha-\frac{p_{n}}{q_{n}}\right|=q_{n}^2\left|\frac{p_{n}\alpha^+_{n}+p_{{n}-1}}{q_{n}\alpha^+_{n}+q_{n-1}}-\frac{p_{n}}{q_{n}}\right| = \frac{q_n}{q_n\alpha^+_{n}+q_{n-1}}=c_n(\alpha),$$
where we have used \eqref{eq:pq} and \eqref{eq:AlphMinusFrac}. 
\end{proof}
Whenever it is clear from context, we will drop the dependence on $\alpha$ of $\Lambda_n(\alpha)$.

\subsection{Ostrowski representation}

Another useful tool will be the fact that one can represent each integer $N\in\NN$ in a unique way with respect to the best approximation denominators of some irrational $\beta$. 
\begin{thmAlph}[Ostrowski representation]\label{OstrowskiRepr}
Let $\beta \in (0,1)$ be an irrational number with continued fraction expansion $[0;b_1, b_2, \ldots]$ and best approximation denominators $(q_n(\beta))_{n\geq 0}$. Then every non-negative integer $t$ has a unique expansion
\begin{equation}
\label{eq:Ostrowski}
t=\sum_{n=1}^N v_n q_n(\beta) ,
\end{equation}
where:
\begin{enumerate}[label=\roman*)]
\item $0\leq v_1< b_1$ and $0\leq v_n \leq b_n$ for $n>1$.
\item If $v_n=b_n$ for some $n$, then $v_{n-1}=0$.
\item $N=N(t)=\mathcal{O}(\log(t))$
\end{enumerate}
We refer to \eqref{eq:Ostrowski} as the Ostrowski representation of $N$ in base $\beta$.
\end{thmAlph}
\noindent 
The proof of Theorem~\ref{OstrowskiRepr} can be found in \cite[p.~126]{KN06}. Recall that we defined (see \eqref{eq:DefDt}) for an irrational $\beta$
\begin{equation*}
	D_t(\beta)= \sum_{s=1}^t\{\beta t\}-\frac{1}{2}.
\end{equation*}
It is possible to give an explicit formula for $D_t(\beta)$ in terms of the digits of the Ostrowski expansion of $t$ in base $\beta$.

\begin{thmAlph}[\cite{BS95}, Theorem~1]\label{thm:DtFormula}
Let $\beta \in (0,1)$ and irrational. For any $t\geq1$ let $t=\sum_{i=1}^N v_i q_i(\beta)$ be the Ostrowski expansion of $t$ in base $\beta$. Then we have
\begin{equation}\label{eq:DtFormula}
	D_t(\beta) = \frac{1}{2} \sum_{i=1}^N\left(v_i\Lambda_{i}(\beta)(v_iq_{i}(\beta)+1)+(-1)^{i}v_i\right) +\sum_{1\leq i<j\leq N}v_iv_jq_{i}(\beta)\Lambda_{j}(\beta).
\end{equation}
\end{thmAlph}
It is well known (see for example \cite{HL21,O22}) that for irrationals $\beta$ with bounded c.f.c. we have that
\begin{equation}\label{eq:Dtbcfc}
	D_t(\beta) = \calO(\log(t)).
\end{equation}
\begin{rem}\label{rem:DtBetaBound}
This also follows from Theorem~\ref{thm:DtFormula} together with the facts that $q_n(\beta)|\Lambda_n(\beta)|=c_n(\beta)$ and that for the Ostrowski representation of $t$ in base $\beta$ i.e. $t= \sum_{i=1}^N v_iq_i(\beta)$ we have $\sum_{i=1}^k v_iq_i(\beta) < q_{k+1}(\beta)$ for all $k \in \{1,\ldots,N-1\}$. Using these facts we obtain for an irrational $\beta=[0;b_1,b_2,\ldots]$ that
\begin{align*}
|D_t(\beta)|&\leq \frac{1}{2} \sum_{i=1}^Nv_i\left|1-c_i(\beta)v_i-|\Lambda_i(\beta)|\right| + \sum_{j=1}^Nv_j|\Lambda_j(\beta)|\sum_{i=1}^{j-1}v_iq_i\\
&\leq \frac{1}{2}\sum_{i=1}^Nv_i + \sum_{i=1}^Nv_i\leq \frac{3}{2} \sum_{i=1}^Nb_i.
\end{align*}
Equation \eqref{eq:Dtbcfc} now follows by the fact that $\beta$ has bounded c.f.c. and since $N=\calO(\log(t))$ (see Theorem~\ref{OstrowskiRepr}).
\end{rem}

In fact it will turn out later on that we need analogs of \eqref{eq:DtFormula} and \eqref{eq:Dtbcfc} for certain rational $\beta$. More precisely, we are interested in the case $\beta_n=\alpha_{n}^- = q_{n-1}(\alpha)/q_{n}(\alpha)$ for a fixed $n\in\NN$. Observe that in this case the Ostrowski representation in base $\beta_n=[0;b_1,\ldots,b_{n-1}]$ of $t\in\{1,\ldots,q_n(\beta_n)-1\}$ is still well defined (here $b_i =a_{n-i}$). One can check, by following the exact same arguments as in \cite{BS95} and Remark~\ref{rem:DtBetaBound}, that the statement in Theorem~\ref{thm:DtFormula} and \eqref{eq:Dtbcfc} are still valid for $\beta_n=\alpha_{n}^- = q_{n-1}(\alpha)/q_{n}(\alpha)$ as long as $t<q_{n}(\beta_n)$.\\
	Another fact that should be mentioned in this context is that 
	$$q_n(\beta_n)=q_n(\alpha_n^-)=q_n(\alpha) \text{ for all } n\in\NN.$$
	This can be seen by using the fact that for $\gamma=[0;c_1,c_2,\ldots]$ we have $q_n(\gamma) = \det(C_{n-1}(\gamma))$ combined with a Laplace expansion argument. Here $C_{n-1}(\gamma) \in \ZZ^{(n-1)\times(n-1)}$ is a tridiagonal matrix where the entries in the first upper minor diagonal are $-1$, the entries in the first lower minor diagonal are $1$ and the main diagonal consists of $c_1,\ldots,c_{n-1}$.
	
	%%%%%%%%%%%%%%%%%%%%%%%%%%%%%%%%%%%%%%%%%%%%%%%%%%%%%%%%%%%%%%%%%%%%%%%%%%%%%%%%%%%%%%%%%%%%%%
	% STRATEGY
	%%%%%%%%%%%%%%%%%%%%%%%%%%%%%%%%%%%%%%%%%%%%%%%%%%%%%%%%%%%%%%%%%%%%%%%%%%%%%%%%%%%%%%%%%%%%%%

	\section{Proof-strategy for Theorems~\ref{thm:PqnBehaviour} and \ref{thm:ConvBoundedCFC}}\label{sec:ProofStrategy}
	
	In \cite{VM16} Mestel and Verschueren were able to find a suitable decomposition of the product $P_{F_n}(\varphi)$, where $\varphi$ is the golden ratio. With the help of this decomposition they managed to prove that the sequence $(P_{F_n}(\varphi))_{n\in\N}$ is convergent, where $(F_n)_{n\in\N}$ are the Fibonacci numbers (note that $q_n(\varphi)=F_n$). Exactly this approach was generalized to arbitrary quadratic irrationals in \cite[Lemma~4.2]{GN18}.	This was one of the key steps in proving the corresponding convergence result for quadratic irrationals (see Theorem~\ref{thm:QuadIrrConv}). Therefore our first step will be to generalize \cite[Lemma~4.2]{GN18} to the case of an arbitrary irrational $\alpha$. This will result in the following lemma.
	\begin{lem}\label{lem:decomp}
		Let $\alpha \in(0,1)$ and irrational. For $n\geq1$ let $p_n/q_n$ be the $n$th convergent of $\alpha$. Then we can rewrite the product  $P_{q_{n}}(\alpha)$ as
		$$P_{q_{n}}(\alpha) = \prod_{r=1}^{q_{n}} |2\sin \pi r \alpha | = A_nB_nC_n ,$$ 
		where
		\begin{align}
			A_n &= \left| 2q_{n} \sin(\pi \Lambda_n)\right|; \label{eq:An}\\
			B_n &= \left| \prod_{t=1}^{q_{n}-1} \frac{s_{nt}}{2 \sin (\pi t / q_{n})} \right|;
				\label{eq:Bn} \\
			C_n &= \prod_{t=1}^{q_{n}-1} \left( 1- \frac{s_{n0}^2}{s_{nt}^2} \right)^{1/2}, 
				\label{eq:Cn}
		\end{align}
		where $\Lambda_n$ is as in \eqref{eq:Lambdan} and for $t\in\{0,1,\ldots,q_n-1\}$ we have set
		\begin{equation}\label{eq:snt}
			s_{nt}:=2\sin\left(\pi \left[\frac{t}{q_n} -|\Lambda_n|\left(\left\{\frac{tq_{n-1}}{q_n}\right\} -\frac{1}{2}\right)\right]\right).
		\end{equation}
	\end{lem}
	
	The second step and also the main part of this article is to understand how the structure of the c.f.e. of the irrational $\alpha$ influences the asymptotic behaviour of $(P_{q_{n}}(\alpha))_{n\in\NN}$. This is will be achieved by a more detailed investigation of the factors $A_n$, $B_n$ and $C_n$ in Lemma~\ref{lem:decomp}. The results of this investigation are described in the subsequent lemma.
	\begin{lem}\label{lem:StructurePqn}
	Let $\alpha\in(0,1)$ be irrational with c.f.e. $\alpha=[0;a_1,a_2,\ldots]$ and let $D_t(\beta)$, $c_n$ and $\Lambda_n$ be defined as in \eqref{eq:DefDt}, \eqref{eq:cn} and \eqref{eq:Lambdan}. Let $(\tau_n)_{n\in\NN}$ and $(\kappa_n)_{n\in\NN}$ be eventually increasing sequences of natural numbers with 
	\begin{equation}\label{eq:taunkappan}
	 \tau_n,\kappa_n = \calO(q^{1/2}_n). 
	\end{equation}
	Then for sufficiently large $n$ we have that
	\begin{equation*}
		P_{q_{n}}(\alpha)=A_nB_nC_n,
	\end{equation*}
	where 
	\begin{align}
		A_n=&2\pi c_{n}(1+\calO(\Lambda_{n}^2));\\
		%
		%\log(B_n) =& \Theta\left(c_{n}R_n(\alpha)\right) -2\sum_{t=1}^{\tau_n}\sum_{j=2}^{\tau_n} \frac{1}{j}\left(\frac{c_{n}\xi_{nt}}{t}\right)^j		+\calO\left(|\Lambda_{n}|^{1/2}\right);\\
		\log(B_n) =& -2 \pi^2\frac{c_n}{q^2_n}\sum_{t=1}^{M_n-1}\frac{D_t(\alpha_n^-)}{\sin(\pi t/q_n)\sin(\pi(t+1)/q_n)} \nonumber\\
		&-2\sum_{t=1}^{\tau_n}\sum_{j=2}^{\tau_n} \frac{1}{j}\left(\frac{c_{n}\xi_{nt}}{t}\right)^j 	+\calO\left(\tau_n^{-1}\right); \label{eq:logBn}\\
		C_n=&\prod_{t=1}^{\kappa_n} \left( 1- \frac{1}{4\left( t/ c_{n}-\xi_{n t} \right)^2}\right)+\calO\left(\kappa_n^{-1}\right), \label{eq:CnPrecise}
	\end{align}
	with 
	\begin{align}
	\xi_{nt} &= \left\{\frac{tq_{n-1}}{q_n}\right\}-\frac{1}{2} \text{ for } t\in\{0,1,\ldots,q_{n}-1\}. \label{eq:xint}
	%\kappa_{n}&=\calO(q_n^{1/2}) \text{ and } \tau_n =\calO(q_n^{1/2}).
	\end{align}
	
\end{lem} 
\begin{rem}
	It should be mentioned at this point that the condition in \eqref{eq:taunkappan} on the eventually increasing sequences $(\kappa_n)_{n\in\NN}$, $(\tau_n)_{n\in\NN}$ could be relaxed to $\tau_n,\kappa_n = \smallCalO(q_n)$. This would lead to remainder terms of order  $\calO(\tau_n^{-1} + \tau_n^2q_n^{-2})$ and $\calO(\kappa_n^{-1} + \kappa_n^2q_n^{-2})$ in \eqref{eq:logBn} and \eqref{eq:CnPrecise}, respectively.	But it is enough for our purposes to stick to the less technical version described in \eqref{eq:taunkappan}.
\end{rem}
We will see later on that the proof of Theorem~\ref{thm:PqnBehaviour} is then a direct consequence of Lemma~\ref{lem:StructurePqn}. 

Another important observation is that in the case of quadratic irrationals the periodicity of the c.f.e. induces a certain structure (see Theorem~\ref{thm:QuadIrrConv}) on the convergent subsequences of $(P_{q_n}(\alpha))_{n\in\NN}$. In the more general case of irrationals with bounded c.f.c. this effect is reflected by the sufficient conditions stated in Theorem~\ref{thm:ConvBoundedCFC}. It will turn out that Lemma~\ref{lem:StructurePqn} together with the requirements of Theorem~\ref{thm:ConvBoundedCFC} will give us the desired convergence results.
	
	%%%%%%%%%%%%%%%%%%%%%%%%%%%%%%%%%%%%%%%%%%%%%%%%%%%%%%%%%%%%%%%%%%%%%%%%%%%%%%%%%%%%%%%%%%%%%%
	% DECOMPOSITION
	%%%%%%%%%%%%%%%%%%%%%%%%%%%%%%%%%%%%%%%%%%%%%%%%%%%%%%%%%%%%%%%%%%%%%%%%%%%%%%%%%%%%%%%%%%%%%%

	\section{Proof of Lemma~\ref{lem:decomp} and Lemma~\ref{lem:StructurePqn}}\label{sec:ProofLemmata}
	The proof of Lemma~\ref{lem:decomp} follows the lines of \cite{VM16} and \cite{GN18}, where Lemma~\ref{lem:decomp} has been shown for the cases $\alpha=\varphi$ (golden ratio) and $\alpha$ being a quadratic irrational, respectively. For the sake of completeness we give the most important steps here.
	
	\begin{proof}[Proof of Lemma~\ref{lem:decomp}]
	Following the exact same steps as in \cite[p.22]{GN18} and using that by definition of $\Lambda_n$ we have $q_n\alpha = p_n+ \Lambda_n$, one can easily check that
	\begin{align*} 
		P_{q_n}^2(\alpha)&=\left(2\sin \pi \Lambda_n \right)^2 \prod_{r=1}^{q_n-1}4 \left(\sin^2\left(\pi 
			r \alpha- \frac{\pi}{2} \Lambda_n\right)-\sin^2\left(\frac{\pi}{2} \Lambda_n\right)\right).
	\end{align*}
	If we further use the identity $\alpha = p_n/q_n +\Lambda_n/q_n$ we obtain
	\begin{align*}
		\sin^2 \left(\pi r \alpha- \frac{\pi}{2} \Lambda_n \right) = \sin^2 \left( \pi\left[\frac
			{rp_{n}}{q_n}+\Lambda_n \left(\frac{r}{q_n}-\frac{1}{2}\right)\right]\right).
	\end{align*}
	By the substitution $t= rp_n \bmod q_n$, and recalling from \eqref{eq:pq} that $p_nq_{n-1} = (-1)^n \bmod q_n$, we have
	\begin{align*}
		\sin^2 \left(\pi r \alpha- \frac{\pi}{2}\Lambda_n \right) &= \sin^2 \left(\pi\left[\frac{rp_{n} \bmod{q_n}}{q_n}+\Lambda_n\left(\frac{r}{q_n}-\frac{1}{2}\right)\right] \right)\\
		&= \sin^2 \left(\pi\left[\frac{t}{q_n}+ \Lambda_n \left(\frac{(-1)^ntq_{n-1} \bmod{q_n}}{q_n}- \frac{1}{2}\right)\right] \right)\\
		&=\frac{1}{4}s_{nt}^2,
	\end{align*}
	with $s_{nt}$ given in \eqref{eq:snt}. Note that we applied the identities $\Lambda_n=(-1)^{n-1}|\Lambda_n|$ (see Lemma~\ref{lem:PropLambda}) and
	$$\frac{(-1)^ntq_{n-1} \bmod{q_n}}{q_n}-\frac{1}{2}= \left\{\frac{(-1)^ntq_{n-1}}{q_n} \right\} - \frac{1}{2} = (-1)^n\left(\left\{\frac{tq_{n-1}}{q_n}\right\} - \frac{1}{2} \right).$$ 
	As $r$ runs through the values $1,2,\ldots , q_n-1$, so does $t=r p_n \bmod q_n$. Accordingly, we get
	\begin{align*}
		P^2_{q_n}(\alpha) &=(2\sin( \pi \Lambda_n ))^2 \prod_{t=1}^{q_n-1}
			 \left( s^2_{nt}-s^2_{n0} \right)\\
		&=(2q_n\sin( \pi \Lambda_n ))^2 \prod_{t=1}^{q_{n}-1}\frac{s^2_{nt}}{4\sin^2(\pi t/ q_n)}\prod_{t=1}^{q_{n}-1} \left(1-\frac{s^2_{n0}}{s^2_{nt}}\right).
	\end{align*}
	For the last equality above we have used the well-known identity
	\begin{equation*}
	\prod_{r=1}^{q-1} 2 \sin \left( \frac{\pi rp}{q} \right) = q 
	\end{equation*}
	whenever $p,q \in \ZZ$ satisfy $\gcd (p,q) =1$ (see e.g.\ \cite{M62} for a nice proof). Taking the square root of both sides, we arrive at
	$$P_{q_n}(\alpha)  = A_nB_nC_n ,$$
	where $A_n$, $B_n$ and $C_n$ are given in \eqref{eq:An}, \eqref{eq:Bn} and \eqref{eq:Cn}, respectively.
	\end{proof}
	
	In order to prove the statement of Lemma~\ref{lem:StructurePqn} we have to investigate each of the factors $A_n$, $B_n$ and $C_n$ more carefully. To do so we will follow again the ideas of \cite{VM16} and \cite{GN18}. This task will be split into three lemmata, one for each of the three factors. We start with the simplest one, the factor $A_n$.
	\begin{lem}\label{lem:AnSeparate}
		Let $A_n$ be as in \eqref{eq:An}. Then for sufficiently large $n$ we have that
		\begin{equation}
			A_n=2\pi c_{n}\left(1+\calO\left(\Lambda_{n}^2\right)\right).
		\end{equation}
	\end{lem}
	\begin{proof}
		Recall that by \eqref{eq:An} we know that $A_n=2\pi q_n \sin(\pi \Lambda_n)$. Since $\Lambda_n$
		tends to 0 for $n\to \infty$ (see Lemma~\ref{lem:PropLambda}), we get by $\sin(x) = x(1+\calO(x^2))$ for $x \to 0$ that for sufficiently large $n$ 
		$$A_n=2 q_n \sin(\pi \Lambda_n) = 2\pi q_n |\Lambda_n| \left(1+\calO(\Lambda_n^2)\right) = 
			2\pi c_n\left(1+\calO(\Lambda_n^2)\right).$$
		Note that for the last step we have used the third property of Lemma~\ref{lem:PropLambda}.
	\end{proof}
	
	Before moving on to the factor $C_n$ recall from \eqref{eq:Cn} that the sequences $s_{nt}$ given in \eqref{eq:snt} and $\xi_{nt}$ given in \eqref{eq:xint} play an important role for the behaviour of $C_n$. Further, it is not hard to see that
	\begin{equation}
		s_{nt}=2 \sin \left(\pi\left[\frac{t}{q_n} - \xi_{nt}|\Lambda_{n}| \right]\right).
	\end{equation}
	The situation is similar for the following sequence $h_{nt}$ and the factor $B_n$
	\begin{equation}\label{eq:hnt}
		h_{nt}:= \cot\left(\frac{\pi t}{q_n}\right)\sin\left( \pi |\Lambda_n|\xi_{nt} \right)\text{ for } t \in \{0,1,\ldots,q_n-1\}.
	\end{equation}
	The exact role of $h_{nt}$ will become more clear once we focus on the factor $B_n$, which happens in the last part of this section. But still it is convenient to introduce the sequence $h_{nt}$ already at this point since otherwise we would have to repeat a similar version of Lemma~\ref{lem:propSequences} later on. We continue with some auxiliary results on $s_{nt}$, $\xi_{nt}$ and $h_{nt}$.
	
	\begin{lem}\label{lem:propSequences}
	Let $s_{nt}$, $\xi_{nt}$ and $h_{nt}$ be the sequences given in \eqref{eq:snt}, \eqref{eq:xint} and \eqref{eq:hnt}. We have that:
	\begin{enumerate}[label=(\alph*)]
		\item $s_{nt} = s_{n(q_{n}-t)}$, $h_{nt} = h_{n(q_{n}-t)}$ and $\xi_{nt}=
			-\xi_{n(q_{n}-t)}$ for all $1\leq t < q_n$.
			\label{it:propsm1}
				
		\item $s_{nt} > s_{n0}$ for all $1\leq t < q_n$ if $n$ is 
			sufficiently large.\label{it:propsm2}
	\end{enumerate}
\end{lem}
\begin{proof}
	We first verify \ref{it:propsm1}. The fact that $\{-x\}= 1-\{x\}$ for $x\in \mbb{R}\setminus
	\mbb{Z}$ immediately implies $\xi_{nt}=-\xi_{n(q_n-t)}$. Combining this with $\sin(\pi - x) = 
	\sin x$, we get
	\begin{align*}
		s_{n(q_n-t)} &= 2 \sin \left(\pi\left[\frac{q_n-t}{q_n} - \xi_{n(q_n-t)}|\Lambda_{n}| \right]\right) \\
		&= 2 \sin \left(\pi\left[ 1- \left( \frac{t}{q_n}- \xi_{nt} |\Lambda_{n}| \right) \right]\right) = 
		s_{nt},
	\end{align*}
	and likewise since $\cot(\pi-x) =-\cot( x)$ and $\sin(x)$ is an odd function, we immediately get by using $\xi_{nt}=-\xi_{n(q_n-t)}$ that $h_{n(q_n-t)}=h_{nt}$
	%\begin{align*}
		%h_{n(q_n-t)} = \cot \left( \pi \left[ 1- \frac{t}{q_n}\right]\right) \sin \left( - \pi 
			%|\Lambda_{n}| \xi_{nt}\right) = h_{nt}
	%\end{align*}
	for every $t \in \{1, \ldots , q_n-1\}$. Now let us verify \ref{it:propsm2}. In light of \ref{it:propsm1}, it is enough to verify 
	$s_{nt} > s_{n0}$ for $t \in \{1,2, \ldots ,  \lfloor q_n/2 \rfloor \}$.
	Writing again 
	$$s_{nt}=2 \sin \left(\pi\left[\frac{t}{q_n} - \xi_{nt}|\Lambda_{n}| \right]\right)$$
	and recalling that $|\xi_{nt}| < 1/2$ for these
	values of $t$, it is clear that for all $t \in\{2,\ldots,\lfloor q_n/2 \rfloor\}$ it follows $s_{nt} > s_{n(t-1)}$, and in particular for $t=1$
	\begin{equation*}
		s_{n1} > 2 \sin \left(\pi \left[ \frac{1}{q_n} -\frac{1}{2}|\Lambda_{n}| \right]\right) > 
		2\sin\left(	\frac{\pi |\Lambda_{n}|}{2}\right) = s_{n0},
	\end{equation*}
	since $|\Lambda_{n}|<1/q_{n+1}<1/q_n$ which was shown in Lemma~\ref{lem:PropLambda}.
\end{proof}
Now we are ready for the analysis of the factor $C_n$. It will be our next goal to prove the subsequent lemma. 

	\begin{lem}\label{lem:CnSeparate}
		Let $C_n$ be as in \eqref{eq:Cn}. For sufficiently large $n$ we have that
		\begin{equation}
			C_n=\prod_{t=1}^{\kappa_n} \left( 1- \frac{1}{4\left( t/ c_{n}-\xi_{n t} \right)^2}\right)+\calO\left(\kappa_n^{-1}\right),
		\end{equation}
		where $\kappa_n= \lfloor q_n^{1/2} \rfloor$.
	\end{lem}
	\begin{rem}\label{rem:differentChoiceKappa}
		Note that Lemma~\ref{lem:CnSeparate} is actually slightly weaker than the statement in Lemma~\ref{lem:StructurePqn} since we set $\kappa_n= \lfloor q_n^{1/2} \rfloor$. The more general version (i.e. $\kappa_n=\calO(q^{1/2}_n)$)  follows by the exact same proof steps. But the choice $\kappa_n= \lfloor q_n^{1/2} \rfloor$ will increase the readability of the subsequent proof.
	\end{rem}
	
\begin{proof}
We begin by developing estimates for $s_{n0}/s_{nt}$, $t\in\{1,\ldots,q_{n}-1\}$. If $n$ is sufficiently large it follows by using $\sin(x)=x(1+\calO(x^2))$ that
\begin{equation*}
s_{n0} = 2 \sin \left( \pi |\Lambda_{n}|/2 \right) = \pi |\Lambda_{n}| \left( 1+\mathcal{O}(\Lambda_{n}^{2})\right).
\end{equation*}
For $t\geq 1$ we have that
\begin{equation}\label{eq:smapprox}
\begin{aligned}
s_{nt} &= 2 \sin \left(\pi \left[ \frac{t}{q_{n}} - |\Lambda_{n}|\xi_{nt} \right]\right) = 2 \sin \left(\pi t |\Lambda_{n}| \left[ \frac{1}{c_n} - \frac{\xi_{nt}}{t}\right]\right),
\end{aligned}
\end{equation}
where we have used Lemma~\ref{lem:PropLambda}. We now split the values of $t$ at 
\begin{equation}\label{eq:DefKappan}
	\kappa_n := \floor{q_n^{1/2}},
\end{equation}	
and treat $t \leq \kappa_n$ and $t > \kappa_n$ separately in order to find appropriate bounds on $s_{nt}$ in \eqref{eq:smapprox}. For $t> \kappa_n$, we apply $\sin(x) \geq 2x/\pi$ for $x \in [0, \pi/2]$ to obtain
\begin{equation*}
s_{nt} \geq 4t|\Lambda_{n}| \left(\frac{1}{c_n}  - \frac{\xi_{n t}}{t} \right)\geq 4t\frac{|\Lambda_{n}|}{c_n} \left(1  - \frac{|\xi_{n t}|c_n}{t} \right) .
\end{equation*}
Note that it follows from Lemma~\ref{lem:PropLambda} that $0< c_n < 1$ and recall that  $|\xi_{n t}| < 1/2$. Thus, for sufficiently large $n$ (and thereby sufficiently large $t$), we have $s_{nt} > 2t |\Lambda_{n}|/c_n$ and 
\begin{equation*}
\frac{s_{n0}}{s_{nt}} \leq \frac{\pi c_n|\Lambda_{n}|\left( 1+\mathcal{O}(\Lambda_{n}^{2}) \right)}{2t|\Lambda_{n}|} = \frac{\pi c_n}{2t} \left( 1+\mathcal{O}(\Lambda_{n}^{2}) \right).
\end{equation*}
Further it follows that 
\begin{equation*}
\sum_{t=\kappa_n+1}^{q_n-1} \frac{s_{n0}^2}{s_{nt}^2} \leq \frac{\pi^2c_n^2}{4}(1+\calO(\Lambda_{n}^2))\sum_{t=\kappa_n+1}^{q_n-1}\frac{1}{t^2} = \calO(\kappa_n^{-1}),
\end{equation*}
and accordingly this sum is convergent and smaller than one for sufficiently large $n$. 
Thus, by Lemma \ref{lem:ProductSumLemma} and for sufficiently large $n$ we get that
\begin{equation}
\label{eq:largeetabound}
1 \geq \prod_{t=\kappa_n+1}^{q_n-1} \left( 1-\frac{s_{n0}^2}{s_{nt}^2} \right)^{1/2} >\left(1- \sum_{t=\kappa_n+1}^{q_n-1} \frac{s_{n0}^2}{s_{nt}^2}\right)^{1/2} = 1-\mathcal{O}(\kappa_n^{-1}),
\end{equation}
where we used the fact that $\sqrt{1-x}\geq 1-x$ for $x\in(0,1)$. Now consider $t \leq \kappa_n.$ %If we set
%\begin{equation}
%\label{eq:ut}
	%u_{n}(t) = 2 \left( \frac{t}{c_n} - \xi_{n t} \right) = 2 \left( \frac{t}{c_n}-\{t\alpha_{n}^-\}+\frac{1}{2} \right)
%\end{equation}
%and follow the exact same steps as in \cite[p.\,26]{GN18} we end up with
It is clear from \eqref{eq:smapprox} that by choosing $n$ sufficiently large, the argument in the sine function $s_{nt}$ can be made arbitrarily small in this case. Applying $\sin(x) = x+\mathcal{O}(x^3)$, we get
\begin{align*}
s_{nt} &= 2\sin\left(\pi\left[\frac{t}{q_{n}} - |\Lambda_{n}|\xi_{nt} \right]\right) =2\pi|\Lambda_{n}|\left(\frac{t}{c_n} - \xi_{nt}
		\right) + \calO(t^3q_n^{-3})\\
	&=\pi |\Lambda_{n}| u_{n}(t)\left(1 + \calO(t^2q_n^{-2})\right)=\pi |\Lambda_{n}| u_{n}(t)\left(1 + \calO(q_n^{-1})\right),
\end{align*}
where we have used that $t\leq \kappa_n\leq q_n^{1/2}$ and $(|\Lambda_n|u_t)^{-1}=\calO(t^{-1}q_n) $. Moreover, we introduced the notation
\begin{equation}
\label{eq:ut}
u_{n}(t) = 2 \left( \frac{t}{c_n} - \xi_{n t} \right) = 2 \left( \frac{t}{c_n}-\{t\alpha_{n}^-\}+\frac{1}{2} \right).
\end{equation}
Thus,
\begin{equation*}
\frac{s_{n0}}{s_{nt}} = \frac{\pi |\Lambda_{n}| \left( 1+\mathcal{O}(\Lambda_{n}^{2})\right)}{\pi |\Lambda_{n}| u_{n}(t)\left(1 + \mathcal{O}(q_n^{-1})\right)} =\frac{1+\mathcal{O}(q_n^{-1})}{u_{n}(t)}
\end{equation*}
%\begin{align}
	%s_{nt} &= 2 \pi |\Lambda_n|\left(\frac{t}{c_n}  - \xi_{nt}\right) +  \calO(|\Lambda_n|^3t^3c_n^{-3})=\pi |\Lambda_n|(u_t +\calO(t^2q_n^{-2})
%\end{align}
and further
\begin{align}
\prod_{t=1}^{\kappa_n}\left( 1-\frac{s_{n0}^2}{s_{nt}^2} \right) %&= \prod_{t=1}^{\kappa_n} \left( 1- 	\frac{1}{u_{n}(t)^2} - \frac{\mathcal{O}(|\Lambda_n|)}{u_{n}(t)^2} \right) \nonumber\\
&= \prod_{t=1}^{\kappa_n} \left( 1-\frac{1}{u_{n}(t)^2}\right) \prod_{t=1}^{\kappa_n} \left( 1- \frac{
	\mathcal{O}(q_n^{-1})}{u_{n}(t)^2-1}\right).\label{eq:CntsmallerKappaPart1}
\end{align}
Let us have a closer look at the two products on the final line above. Since $|\xi_{n t}|<1/2$ and $c_n<1$, we see from \eqref{eq:ut} that $u_n(t) >1$ for all $1 \leq t \leq \kappa_n$. Therefore both products are well-defined and we see that 
$$\sum_{t=1}^{\kappa_n}\frac{1}{u_{n}(t)^2-1} \leq \sum_{t=1}^{\infty}\frac{1}{4(t/c_n-1/2)^2-1}<\infty.$$
Hence, $\sum_t \mathcal{O}(q_n^{-1})/(u_{n}(t)^2-1) = \mathcal{O}(q_n^{-1})$. The latter sum is thus smaller than one, provided $n$ is sufficiently large, and again it follows from Lemma~\ref{lem:ProductSumLemma} that
\begin{equation}
\label{eq:CntsmallerKappaPart2}
1 > \prod_{t=1}^{\kappa_n} \left( 1- \frac{\mathcal{O}(q_n^{-1})}{u_{n}(t)^2-1}\right) \geq 1- \sum_{t=1}^{\kappa_n} \frac{\mathcal{O}(q_n^{-1})}{u_{n}(t)^2-1}  = 1- \mathcal{O}(q_n^{-1}).
\end{equation}
Combining the estimates \eqref{eq:largeetabound} with \eqref{eq:CntsmallerKappaPart1} and  \eqref{eq:CntsmallerKappaPart2} we obtain that 
\begin{align*}
C_n &= \prod_{t=1}^{\kappa_n}\left( 1- \frac{1}{u_{n}(t)^2} \right) \cdot \prod_{t=1}^{\kappa_n} \left( 1- \frac{\mathcal{O}(q_n^{-1})}{u_{n}(t)^2-1}\right) \cdot \prod_{t=\kappa_n+1}^{(q_n-1)/2} \left( 1- \frac{s_{n0}^2}{s_{nt}^2} \right)\\
&=\prod_{t=1}^{\kappa_n}\left( 1- \frac{1}{u_{n}(t)^2} \right) + \calO(\kappa_n^{-1}).
\end{align*}
This completes the proof of Lemma~\ref{lem:CnSeparate}.
\end{proof} 

We continue with the factor $B_n$. What remains to be shown for Lemma\,\ref{lem:StructurePqn} is the following statement.
\begin{lem}\label{lem:BnSeparate}
	Let $B_n$ be as in \eqref{eq:Bn}, $M_n=\floor{(q_n-1)/2}$. Then for sufficiently large $n$ we have that
	\begin{align}
		\log(B_n) =& -2 \pi^2\frac{c_n}{q_n^2}\sum_{t=1}^{M_n-1}\frac{D_t(\alpha_n^-)}{\sin(\pi t/q_n)\sin(\pi(t+1)/q_n)} \nonumber\\
		&-2\sum_{t=1}^{\tau_n}\sum_{j=2}^{\tau_n} \frac{1}{j}\left(\frac{c_{n}\xi_{nt}}{t}\right)^j 	+\calO\left(\tau_n^{-1}\right),
	\end{align}
	where $\alpha_n^-$ and $D_t(\alpha_n^-)$ are given in \eqref{eq:defAlphPlusAlphMinus} and \eqref{eq:DefDt}, respectively and $\tau_n=\lfloor q_n^{1/2} \rfloor$.
\end{lem}
\begin{rem}\label{rem:ChoiceOftaun}
	Similar as for the factor $C_n$ and Lemma~\ref{lem:BnSeparate} we again prove a slightly simpler version ($\tau_n=\lfloor q_n^{1/2} \rfloor$) compared to the statement in Lemma~\ref{lem:StructurePqn} ($\tau_n=\calO(q_n^{1/2})$). Again the more general version of Lemma~\ref{lem:StructurePqn} can be proved by following the exact same steps as in the subsequent proof.
\end{rem}
\begin{proof}
Recall that we have
\begin{equation*}
B_n = \left| \prod_{t=1}^{q_{n}-1} \frac{s_{nt}}{2 \sin (\pi t / q_{n})} \right|.
\end{equation*}
We begin by examining each term of the product $B_n$. Recalling the definition of $s_{nt}$ from \eqref{eq:snt}, we have
\begin{align*}
\frac{s_{nt}}{2\sin (\pi t/q_n)} &= \cos (\pi |\Lambda_{n}| \xi_{nt}) - \cot (\pi t/q_n) \sin \left(\pi |\Lambda_{n}| \xi_{nt}\right) \\
&= 1- 2\sin^2(\pi |\Lambda_{n}| \xi_{nt}/2) - h_{nt},
\end{align*}
with $h_{nt}$ given in \eqref{eq:hnt}. Taking $\beta_{nt}:= 2\sin^2(\pi |\Lambda_{n}| \xi_{nt}/2)$, it is easily verified that $\beta_{n(q_n-t)} = \beta_{nt}$ for $t \in \{1, \ldots , q_n-1 \}$. Likewise, we recall from Lemma~\ref{lem:propSequences}~\ref{it:propsm1} that $h_{n(q_n-t)}=h_{nt}$, and thus
\begin{equation*}
B_n = \prod_{t=1}^{q_n-1} (1-\beta_{nt}-h_{nt}) = \theta_n\prod_{t=1}^{\floor{(q_n-1)/2}} (1-\beta_{nt}-h_{nt})^2,
\end{equation*}
where $\theta_n=1$ if $q_n$ is odd and $\theta_n=1-\beta_{n(q_n/2)}$ otherwise. (Note that $h_{n(q_n/2)}=0$.)
This shows that we only need to consider $t \in \{1, \ldots , \lfloor(q_n-1)/2\rfloor\}$. 

Let us now show that rather than analyzing $B_n$, we may choose to analyze the simpler product
\begin{equation}
\label{eq:Bstar}
B_n^* := \prod_{t=1}^{q_n-1} (1-h_{nt}) = \prod_{t=1}^{\floor{(q_n-1)/2}} (1-h_{nt})^2.
\end{equation}
It can be shown completely analogously to \cite[p.\,29]{GN18} that we have

%%%%%%%
% BEGIN COMMENTING OUT
%%%%%%%
\begin{comment}
Taking logarithms, we get
\begin{equation}
\label{eq:logsBm}
\log (1-\beta_{nt}-h_{nt}) = \log (1-h_{nt})+ \log \left( 1- \frac{\beta_{nt}}{1-h_{nt}}\right).
\end{equation}
Our claim is that the latter term on the right hand side in \eqref{eq:logsBm} will not contribute significantly to the sum $\log B_n = \sum_t \log (1- \beta_{nt}-h_{nt})$. To see this, let us first estimate the size of $h_{nt}$ and $\beta_{nt}$. Considering only $t \in \{1, \ldots , (q_n-1)/2\}$, we use $\cot (x) < 1/x$ and $\sin(x) < x$ to obtain
\begin{equation}\label{eq:boundhmt}
|h_{nt}| = \cot (\pi t /q_{n}) \sin (\pi |\Lambda_{n}| \xi_{nt}) \leq \frac{c_{n}\xi_{nt}}{t},
\end{equation}
where we have used Lemma~\ref{lem:PropLambda}~\ref{it:LambdaThird} in the last step.
We recall that $c_n \leq 1$. Since $|\xi_{nt}|<1/2$, we thus get for sufficiently large $n$ that
\begin{equation}
\label{eq:hmbound}
|h_{nt}|< \frac{1}{2t} < \frac{1}{2},
\end{equation}
and it follows that $1-h_{nt}>1/2$. For $\beta_{nt}$, we have
\begin{equation*}
\beta_{nt} <  2 \left( \frac{\pi |\Lambda_{n}|\xi_{nt}}{2} \right)^2 < \frac{\pi^2\Lambda^2_{n}}{8},
\end{equation*}
and thus for sufficiently large values of $n$ we get $|\beta_{nt}/(1-h_{nt})|<1$ and
\begin{equation}
\label{eq:logzero}
\log \left( 1- \frac{\beta_{nt}}{1-h_{nt}}\right) = - \sum_{j=1}^{\infty} \frac{1}{j} \left( \frac{\beta_{nt}}{1-h_{nt}} \right)^j = \mathcal{O}(\Lambda_{n}^{2}) .
\end{equation}
%%%%%%%
% END COMMENTING OUT
%%%%%%%
\end{comment}

\begin{equation}\label{eq:Bnstar}
\left| \log (B_n) - \log (B_n^*) \right| = \left|\log(\theta_n)+\calO(q_n\Lambda_n^2)\right| = \mathcal{O}(|\Lambda_{n}|),
\end{equation}
where we have additionally used that $|\log(\theta_n)|\leq \beta_{n(q_n/2}) \leq \pi^2/8 \Lambda_n^2$ and $|\Lambda_n|q_n =c_n\leq1$.
 Thus $\lim_{n \to \infty} \log (B_n) = \lim_{n \to \infty} \log (B_n^*)$. This confirms that we may choose to analyze $B_n^*$ in \eqref{eq:Bstar} rather than $B_n$. 

Finally,  we set $M_n:=\floor{(q_n-1)/2}$ and rewrite $\log (B_n^*)$ using its Taylor expansion as
\begin{equation}
\label{eq:Hmsplit}
\begin{aligned}
\log (B_n^*) &= 2 \sum_{t=1}^{M_n} \log (1-h_{nt}) = -2 \sum_{t=1}^{M_n} \sum_{j=1}^{\infty} \frac{1}{j} h_{nt}^j \\
&=-2 \left( \sum_{t=1}^{M_n} h_{nt} + \sum_{t=1}^{M_n} \sum_{j=2}^{\infty} \frac{1}{j}h_{nt}^j \right) =: -2(H_n^{(1)}+H_n^{(2)}).
\end{aligned}
\end{equation}
We continue by investigating the sums $H_n^{(1)}$ and $H_n^{(2)}$ separately in the following subsections.

%%%%%%%%%%%%%%%%%%%%%%%%% H_2-convergence %%%%%%%%%%%%%%%%%%%%%%%%%%%%%%%%
\ \\
\noindent
\textbf{Analysis of $H_n^{(2)}$}:\\
We first treat the sum
\begin{equation*}
H_n^{(2)} =  \sum_{t=1}^{M_n} \sum_{j=2}^{\infty} \frac{1}{j}h_{nt}^j .
\end{equation*}
Let us start by proving that terms where $t$ or $j$ is greater than some $\tau_n$ will not contribute significantly to $H_n^{(2)}$. Note that by applying $\cot (x) = 1/x(1+ \calO(x^2))$ and $\sin(x) =x(1+\calO(x^2))$ we obtain for sufficiently large $n$
\begin{equation}\label{eq:hntBound}
	h_{nt} = \cot (\pi t /q_{n}) \sin (\pi |\Lambda_{n}| \xi_{nt}) = \frac{c_{n}\xi_{nt}}{t}(1+\calO(t^2q_n^{-2}))
\end{equation}
and if we use the rougher bounds $\cot(x)<1/x$ and $\sin(x)<x$ we get for sufficiently large $n$ that
\begin{equation}\label{eq:hntBoundRough}
	h_{nt} = \cot (\pi t /q_{n}) \sin (\pi |\Lambda_{n}| \xi_{nt}) \leq \frac{c_{n}\xi_{nt}}{t}<\frac{1}{2t}.
\end{equation}
Now for $J\geq 2$ we obtain by \eqref{eq:hntBoundRough}
\begin{equation}\label{eq:HelpSum}
\left|\sum_{j=J}^{\infty}\frac{1}{j}h^j_{nt}\right|\leq\sum_{j=J}^{\infty}|h^j_{nt}| =
		\frac{|h^J_{nt}|}{1-|h_{nt}|}<2\left(\frac{1}{2t}\right)^J.
\end{equation}
Further we set 
\begin{equation}\label{eq:DefTaun}
	\tau_n=\lfloor q_n^{1/2} \rfloor
\end{equation}
and therefore we have $1 \leq \tau_n \leq \floor{(q_n-1)/2}=M_n$ if $n$ is sufficiently large. We then get by \eqref{eq:HelpSum}
\begin{equation*}
	\left|\sum_{t=\tau_n+1}^{M_n}\sum_{j=2}^{\infty}\frac{1}{j}h^j_{nt}\right| <\sum_{t=\tau_n+1}^
		{M_n} 2\left(\frac{1}{2t}\right)^2 < \frac{1}{2}\sum_{t=\tau_n+1}^{\infty}\frac{1}{t^2} 
		< \frac{1}{2\tau_n}
\end{equation*}
and again by \eqref{eq:HelpSum}
\begin{equation*}
	\left|\sum_{t=1}^{\tau_n}\sum_{j=\tau_n+1}^{\infty}\frac{1}{j}h^j_{nt}\right| < \sum_{t=1}^{\tau_n}
		2\left(\frac{1}{2t}\right)^{\tau_n+1} < \left(\frac{1}{2}\right)^{\tau_n} \sum_{t=1}^{\infty}
		\frac{1}{t^2} = \frac{\pi^2}{6 \cdot 2^{\tau_n}}.
\end{equation*}
Both of these sums are at most of order $\mathcal{O}(\tau_n^{-1})$, and it follows by \eqref{eq:hntBound}
\begin{align}
	H_n^{(2)} &=  \sum_{t=1}^{\tau_n}\sum_{j=2}^{\tau_n} \frac{1}{j}h^j_{nt} + \sum_{t=1}^{\tau_n}\sum_{j=\tau_n+1}^{\infty}\frac{1}{j}h^j_{nt} + \sum_{t=\tau_n+1}^{M_n}\sum_{j=2}^{\infty}	\frac{1}{j}h^j_{nt}\nonumber\\
	&= \sum_{t=1}^{\tau_n}\sum_{j=2}^{\tau_n} \frac{1}{j}h^j_{nt} + \mathcal{O}(\tau_n^{-1})\nonumber\\
	&= \sum_{t=1}^{\tau_n}\sum_{j=2}^{\tau_n} \frac{1}{j}\left(\frac{c_n\xi_{nt}}{t}\right)^j\left(1+\calO(t^2q_n^{-2})\right)^j + \mathcal{O}(\tau_n^{-1}).
	\end{align}
	%where we have applied in the last step $\cot(x)=1/x(1+\calO(x^2))$ and $\sin(x)=x(1+\calO(x^2))$ to $h_{nt} = \cot(\pi t/q_n)\sin(\pi|\Lambda_n|\xi_{nt})$ and used that $q_n|\Lambda_n|=c_n.$
	Now observe that for some $x >0$
	\begin{equation}\label{eq:helpBound}
	(1+x)^j =1+ \sum_{i=1}^j \binom{j}{i} x^{i} \leq  1+ 2^j x \frac{1-x^j}{1-x}.
	\end{equation}
	Moreover, note that by $t^2<\tau_n^2<q_n^{-1}$ it follows for sufficiently large $n$ that $\calO(t^2q_n^{-2})<1$ and by \eqref{eq:helpBound} that $(1+\calO(t^2q_n^{-2}))^j= 1+ \calO(2^jt^2q_n^{-2})$. Hence,
	\begin{align}
	H_n^{(2)} &=\sum_{t=1}^{\tau_n}\sum_{j=2}^{\tau_n} \frac{1}{j}\left(\frac{c_n\xi_{nt}}{t}\right)^j+\calO\left(\sum_{t=1}^{\tau_n}\sum_{j=2}^{\tau_n} \frac{1}{j}\left(\frac{c_n\xi_{nt}}{t}\right)^j2^jt^2q_n^{-2}\right) + \mathcal{O}(\tau_n^{-1})\nonumber\\
	&=\sum_{t=1}^{\tau_n}\sum_{j=2}^{\tau_n} \frac{1}{j}\left(\frac{c_n\xi_{nt}}{t}\right)^j+\calO\bigg(q_n^{-2}\underbrace{\sum_{t=1}^{\tau_n}\sum_{j=2}^{\tau_n}t^{2-j}}_{\leq 3\tau_n}\bigg) + \mathcal{O}(\tau_n^{-1})\nonumber\\
	&=\sum_{t=1}^{\tau_n}\sum_{j=2}^{\tau_n} \frac{1}{j}\left(\frac{c_n\xi_{nt}}{t}\right)^j+\mathcal{O}(\tau_n^{-1})\label{eq:StructureHn2},
\end{align}
where we have used that $|c_n\xi_n{t}| \leq 1/2$ and $\tau_n\leq  q^{1/2}_n$.
Let us now move on to the analysis of $H_n^{(1)}$.

%%%%%%%%%%%%%%%%%     H_1-convergence    %%%%%%%%%%%%%%%%%%%%%%%%%%%%%%%%

\ \\
\noindent
\textbf{Analysis of $H_n^{(1)}$}:\\
We are left with investigating the behaviour of 
\begin{equation*}
H_n^{(1)} = \sum_{t=1}^{M_n} h_{nt} = \sum_{t=1}^{M_n}  \cot \left( \frac{\pi t}{q_n}\right) \sin (\pi |\Lambda_{n}| \xi_{n t}).
\end{equation*}
It follows by a summation by parts argument (see \cite[p.\,32/33]{GN18}) that

%%%%%%
% BEGIN COMMENTING OUT
%%%%%%
\begin{comment}
\begin{equation}
\label{SummationByParts}
\begin{aligned}
H_n^{(1)} &= \sum_{t=1}^{M_n-1}\left( \cot \left( \frac{\pi t}{q_n} \right) -\cot \left( \frac{\pi (t+1)}{q_n}\right)	\right) \sum_{s=1}^t \sin (\pi |\Lambda_{n}| \xi_{n s}) \\
	&+ \cot \left( \frac{\pi M_n}{q_n} \right) \sum_{s=1}^{M_n} \sin (\pi |\Lambda_{n}|
		\xi_{n s}). 
\end{aligned}
\end{equation}
Consider the second term on the right hand side in this equation. As $|\xi_{n s}|<1/2$ and $\sin (x) = x(1+\mathcal{O}(x^2))$, we have
\begin{equation*}
\left| \sum_{s=1}^{M_n} \sin (\pi |\Lambda_{n}| \xi_{n s}) \right| <\frac{\pi}{2} q_n|\Lambda_n| (1+\mathcal{O}(\Lambda_{n}^2)) = \frac{\pi}{2} c_n (1+\mathcal{O}(\Lambda_{n}^2)) ,
\end{equation*}
where we have again used that $M_n<q_n$ and Lemma~\ref{lem:PropLambda}. It follows that
\begin{equation*}
\left|\cot \left( \frac{\pi M_n}{q_n} \right) \sum_{s=1}^{M_n} \sin (\pi |\Lambda_{n}| \xi_{n s}) \right| < \frac{\pi}{2}\left| c_n \cot \left( \frac{\pi M_n}{q_n} \right) (1+\mathcal{O}(\Lambda_{n}^{2}))\right|, 
\end{equation*}
and recalling that $M_n=\lfloor (q_n-1)/2 \rfloor$, it is clear that the cotangent term tends to zero as $n\rightarrow \infty$. It thus follows from \eqref{SummationByParts} that
%%%%%%
% END COMMENTING OUT
%%%%%%
\end{comment}

\begin{equation}
\label{eq:simCS}
H_n^{(1)} = \sum_{t=1}^{M_n-1} C_{nt} S_{nt} + \calO\left(q_n^{-1}\right),
\end{equation}
where we set $M_n=\floor{(q_n-1)/2}$ and
\begin{align*}
	&C_{nt}:= \cot(\pi t/q_n) - \cot (\pi(t+1)/q_n);\\
	&S_{nt}:=\sum_{s=1}^t \sin( \pi |\Lambda_{n}| \xi_{n s}).
\end{align*}
Let us investigate $C_{nt}$ closer. For ease of notation we write $\phi=\pi/q_n$. Then we have
\begin{align}
	0 < C_{nt} &= \frac{\sin((t+1)\phi)\cos(t\phi) - 
		\cos((t+1)\phi)\sin(t\phi)}{\sin(t\phi)\sin((t+1)\phi)} \nonumber\\
	&= \frac{\sin(\phi)}{\sin(t\phi)\sin((t+1)\phi)}=\frac{\phi \left(1+\calO\left(\phi^2\right)\right)}{\sin(t\phi)\sin((t+1)\phi)}.  \label{eq:BoundCnt}
\end{align}
Thus, by using again $\sin(x)=x(1+\calO(x^2))$ we obtain
\begin{align}\label{eq:CntSntAsymptotics}
	C_{nt}S_{nt} &= \frac{\pi }{q_n\sin(\pi t/q_n)\sin(\pi(t+1)/q_n)} \left(1+\calO\left(\phi^2\right)\right)\sum_{s=1}^t \sin( \pi |\Lambda_{n}| \xi_{n s})\nonumber\\
	&= \frac{\pi^2 |\Lambda_n| }{q_n\sin(\pi t/q_n)\sin(\pi(t+1)/q_n)} \left(1+\calO\left(q_n^{-2}\right)\right)\left(1+\calO(\Lambda_n^2)\right)\sum_{s=1}^t \xi_{n s}\nonumber\\
	&=  \frac{\pi^2 c_n}{q^2_n}\frac{ D_t(\alpha_n^{-}) }{\sin(\pi t/q_n)\sin(\pi(t+1)/q_n)} \left(1+\calO\left(q_n^{-2}\right)\right).
\end{align}
Note that we have used Lemma~\ref{lem:PropLambda} and \eqref{eq:DefDt} in the above calculations. Further, observe that $0<(t+1)/q_n <1/2$ for $t<M_n=\floor{(q_n-1)/2}$. Now it follows by combining the trivial bounds $|D_t(\alpha_n^-)|<t/2$, $c_n\leq 1$ and $\sin(x)\geq 2x/\pi$ for $x\in(0,\pi/2)$ with \eqref{eq:CntSntAsymptotics} that 
$$ \left| \frac{\pi^2 c_n}{q^2_n}\frac{ D_t(\alpha_n^{-}) }{\sin(\pi t/q_n)\sin(\pi(t+1)/q_n)}\right| =  \frac{\pi^2 c_n}{q^2_n}\frac{ \left|D_t(\alpha_n^{-})\right| }{\sin(\pi t/q_n)\sin(\pi(t+1)/q_n)} \leq \frac{\pi^2}{8t}.$$
Therefore it follows from \eqref{eq:CntSntAsymptotics}
\begin{equation}\label{eq:CntSntFinalAsymptotics}
	\sum_{t=1}^{M_n-1}C_{nt}S_{nt} =  \frac{\pi^2 c_n}{q_n^2}\sum_{t=1}^{M_n-1}\frac{ D_t(\alpha_n^{-}) }{\sin(\pi t/q_n)\sin(\pi(t+1)/q_n)} +\calO\left(\frac{\log(q_n)}{q_n^{2}}\right).
\end{equation}
Combining \eqref{eq:Hmsplit}, \eqref{eq:StructureHn2}, \eqref{eq:simCS} and \eqref{eq:CntSntFinalAsymptotics} finally gives for sufficiently large~$n$
\begin{align}\label{eq:BnFinal}
	\log(B_n)=& -2 \frac{\pi^2 c_n}{q_n^2}\sum_{t=1}^{M_n-1}\frac{ D_t(\alpha_n^{-}) }{\sin(\pi t/q_n)\sin(\pi(t+1)/q_n)} \nonumber\\
	&-2 \sum_{t=1}^{\tau_n}\sum_{j=2}^{\tau_n} \frac{1}{j}\left(\frac{c_n\xi_{nt}}{t}\right)^j
	+\mathcal{O}(\tau_n^{-1}).
\end{align}
This finishes the proof of Lemma~\ref{lem:BnSeparate}.
\end{proof}
The statement of Lemma~\ref{lem:StructurePqn} now follows by the combined assertions of Lemma~\ref{lem:AnSeparate}, Lemma~\ref{lem:CnSeparate} and Lemma~\ref{lem:BnSeparate}.

	%%%%%%%%%%%%%%%%%%%%%%%%%%%%%%%%%%%%%%%%%%%%%%%%%%%%%%%%%%%%%%%%%%%%%%%%%%%%%%%%%%%%%%%%
	%%%
	%%%										Proof of Theorem 1.2 and Theorem 1.7 
	%%%
	%%%%%%%%%%%%%%%%%%%%%%%%%%%%%%%%%%%%%%%%%%%%%%%%%%%%%%%%%%%%%%%%%%%%%%%%%%%%%%%%%%%%%%%%

	\section{Proof of Theorem~\ref{thm:PqnBehaviour} and  Theorem~\ref{thm:ConvBoundedCFC}}\label{sec:ProofMainThm}
	In this final section we will give proofs for Theorem~\ref{thm:ConvBoundedCFC} and Theorem~\ref{thm:PqnBehaviour}, starting with the latter one. The proofs of Corollary~\ref{cor:PqnBehaviourTypeOfCFC} and Corollary~\ref{cor:behaviourEuler} are stated in this section as well.

	%%%%%%%%%%%%%%%%%% Proof of Theorem 1.2 %%%%%%%%%%%%%%%%%%%%%%%%%%%%%%%%%%%%%%%%%%%%%%%%%%%%

	\begin{proof}[Proof of Theorem~\ref{thm:PqnBehaviour}]
	Lemma~\ref{lem:StructurePqn} already gives us detailed information concerning the structure of $P_{q_n}(\alpha)$ for an arbitrary irrational $\alpha$. Recall that we have
	\begin{equation}\label{eq:PqnEqualAnBnCn}
		P_{q_n}(\alpha) = A_nB_nC_n,
	\end{equation}
	where
	\begin{align*}
		A_n =& 2\pi c_{n}(1+\calO(\Lambda_{n}^2));\\
		\log(B_n) =& -2 \pi^2\frac{c_n}{q_n^2}\sum_{t=1}^{M_n-1}\frac{D_t(\alpha_n^-)}{\sin(\pi t/q_n)\sin(\pi(t+1)/q_n)} \\ 
		 &-2\sum_{t=1}^{\tau_n}\sum_{j=2}^{\tau_n} \frac{1}{j}\left(\frac{c_{n}\xi_{nt}}{t}\right)^j 	+\calO\left(\tau_n^{-1}\right);\\
		C_n = &\prod_{t=1}^{\kappa_n} \left( 1- \frac{1}{4\left( t/ c_{n}-\xi_{n t} \right)^2}\right)+\calO\left(\kappa_n^{-1}\right)
	\end{align*}
	and $(\kappa_n)_n{\in\NN}$ and $(\tau_n)_n{\in\NN}$ are eventually increasing sequences of natural numbers satisfying \eqref{eq:taunkappan}.
	
	Clearly we have that 
	\begin{equation}\label{eq:AnTheta}
	A_n=\Theta(c_n) \text{ for } n\to\infty.	
	\end{equation}
	For the factor $C_n$ recall that $1/c_n =(\alpha_n^+ + \alpha_n^-)\in (a_n,a_n+2)$ and $|\xi_{nt}| < 1/2.$ It is obvious that $C_n$ can be bounded from above since
	\begin{equation}\label{eq:CnUpperBound}
		 \prod_{t=1}^{\kappa_n} \left( 1- \frac{1}{4\left( t/ c_{n}-\xi_{n t} \right)^2}\right) \leq \prod_{t=1}^{\kappa_n} \left( 1- \frac{1}{\left( 2t(a_n+2)+1 \right)^2}\right) <1.
	\end{equation}
	For the lower bound note that 
	\begin{equation*}
		\sum_{t=2}^{\kappa_n}\frac{1}{4(t/c_n - \xi_{nt})^2} \leq \sum_{t=2}^{\kappa_n}
			\frac{1}{4(ta_n - \frac{1}{2})^2} \leq \sum_{t=2}^{\infty} \frac{1}{(2t-1)^2} <\frac{1}{4}.
	\end{equation*}
	Therefore we can make use of Lemma~\ref{lem:ProductSumLemma} and obtain that
	\begin{align}\label{eq:CnLowerBound}
		\prod_{t=1}^{\kappa_n} \left( 1- \frac{1}{4\left( t/ c_{n}-\xi_{n t} \right)^2}\right) &= \left(1-\frac{1}{4(1/c_n-\xi_{n1})^2}\right)\prod_{t=2}^{\kappa_n} \left( 1- \frac{1}{4\left( t/ c_{n}-\xi_{n t} \right)^2}\right)\nonumber\\
		&\geq\left(1-\frac{1}{(2a_n+1)^2}\right)\left( 1- \sum_{t=2}^{\infty}\frac{1}{4(t/c_n - \xi_{nt})^2}\right)\nonumber\\
		&> \frac{3}{4}\left(1-\frac{1}{(2a_n+1)^2}\right)\geq \frac{2}{3}.
	\end{align}
	Thus, by \eqref{eq:CnUpperBound} and \eqref{eq:CnLowerBound} we get that 
	\begin{equation}\label{eq:CnTheta}
		C_n=\Theta(1) \text{ for } n\to\infty. 
	\end{equation}
	Let us now investigate the remaining factor $B_n$. Observe that the double sum in \eqref{eq:logBn} is actually bounded. We have that
	\begin{align}\label{eq:BoundBnFirst}
		\left|\sum_{t=1}^{\tau_n}\sum_{j=2}^{\tau_n} \frac{1}{j}\left(\frac{c_{n}\xi_{nt}}{t}\right)^j\right| \leq \frac{1}{2}\sum_{t=1}^{\infty}\sum_{j=2}^{\infty} \frac{1}{(2t)^j}= \frac{1}{4}\sum_{t=1}^{\infty}\frac{1}{t(2t-1)}=\frac{\log(2)}{2}.
	\end{align}
	It follows by \eqref{eq:BoundBnFirst} and $\sin(x)=x\left(1+\calO(x^2)\right)$ that
	\begin{align}\label{eqLogBnUnboundedCFC}
		\log(B_n) &= -2 \pi^2\frac{c_n}{q^2_n}\sum_{t=1}^{M_n-1}\frac{D_t(\alpha_n^-)}{\sin(\pi t/q_n)\sin(\pi(t+1)/q_n)} + \calO(1) \nonumber\\
		&=-2 c_n\sum_{t=1}^{M_n-1}\frac{D_t(\alpha_n^-)}{t(t+1)} \left(1+\calO(t^2q_n^{-2})\right) + \calO(1) \nonumber\\
		&= -2 c_n\sum_{t=1}^{M_n-1}\frac{D_t(\alpha_n^-)}{t(t+1)} +\calO\left(1+ 2c_nq_n^{-2}\sum_{t=1}^{M_n-1}D_t(\alpha_n^-)\right)\nonumber\\
		&=-2 c_n\sum_{t=1}^{M_n-1}\frac{D_t(\alpha_n^-)}{t(t+1)} +\calO(1)\nonumber.
	\end{align}
	Hence,
	\begin{equation}\label{eq:BnTheta}
		B_n = \Theta\left(\exp\Big( - 2 c_n\sum_{t=1}^{M_n-1}\frac{D_t(\alpha_n^-)}{t(t+1)} \Big)\right) \text{ for } n \to \infty.
	\end{equation}
	Therefore we obtain from \eqref{eq:PqnEqualAnBnCn} together with \eqref{eq:AnTheta}, \eqref{eq:CnTheta} and \eqref{eq:BnTheta} that
	\begin{equation}\label{eq:PqnThetaBn}
		 P_{q_n}(\alpha) = \Theta\left(c_n\exp\Big( - 2 c_n\sum_{t=1}^{M_n-1}\frac{D_t(\alpha_n^-)}{t(t+1)} \Big)\right) \text{ for } n \to \infty.
	\end{equation}
	This finishes the proof of Theorem~\ref{thm:PqnBehaviour}.
	\end{proof}
	
	Now that we have established Theorem~\ref{thm:PqnBehaviour} we are ready to prove Corollary~\ref{cor:PqnBehaviourTypeOfCFC}.
	\begin{proof}[Proof of Corollary~\ref{cor:PqnBehaviourTypeOfCFC}]
	If $\alpha=[0;a_1,a_2,\ldots]$ has bounded c.f.c. then clearly we have
	\begin{equation}\label{eq:cnBoundedCFC}
		\frac{1}{A+2}\leq \frac{1}{a_n+2}< c_n < \frac{1}{a_n} \leq 1,
	\end{equation}
	where $A:=\sup_{n\in\NN}a_n$. Recall that for $\alpha_n^-=q_{n-1}(\alpha)/q_n(\alpha)$ it was pointed out in Section~\ref{sec:Prelim} that
	\begin{equation*}
		D_t(\alpha_n^-)=\calO(\log(t))
	\end{equation*}
	as long as $t\in\{1,\ldots,q_n(\alpha)-1\}$. Note that $M_n=\floor{(q_n(\alpha)-1)/2}<q_n(\alpha)$ and therefore
	%Note that we have $q_n(\alpha_n^-)=q_n(\alpha)$.
	%This can be seen by using the fact that for $\gamma=[0;c_1,c_2,\ldots]$ we have $q_n(\gamma) = \det(C_{n-1}(\gamma))$ combined with a Laplace expansion argument. Here $C_{n-1}(\gamma) \in a \ZZ^{(n-1)\times(n-1)}$ is a tridiagonal matrix where the entries in the first off diagonal are $1$, the entries in the second off diagonal are $-1$ and the main diagonal consists of $c_1,\ldots,c_{n-1}$. Recall that $M_n-1<q_n(\alpha)$. We deduce that
	\begin{equation}\label{eq:BoundSumDt}
		\left|\sum_{t=1}^{M_n-1}\frac{D_t(\alpha_n^-)}{t(t+1)}\right| \leq K\sum_{t=1}^{\infty}\frac{\log(t)}{t^2} <\infty,
	\end{equation}
	where $K>0$ is some constant independent of $n$. From \eqref{eq:cnBoundedCFC} and \eqref{eq:BoundSumDt} together with Theorem~\ref{thm:PqnBehaviour} it follows that
	$$P_{q_n}(\alpha) = \Theta\left(c_n\exp\Big( - 2 c_n\sum_{t=1}^{M_n-1}\frac{D_t(\alpha_n^-)}{t(t+1)} \Big)\right) = \Theta(1) .$$
	If $\alpha$ has unbounded c.f.c. then we have that $c_n = \Theta(1/a_n)$ and by Theorem~\ref{thm:DtFormula} it follows for $t\in\{1,\ldots,q_n(\alpha)-1\}$ with Ostrowski representation in base $\alpha_n^-$ of the form $t=\sum_{i=1}^{N(t)}v_iq_i(\alpha_n^-)$ that
	\begin{align*}
		D_t(\alpha_n^{-}) &= \frac{1}{2} \sum_{i=1}^{N(t)}\left(v_i\Lambda_{i}(\alpha_n^-)(v_iq_{i}(\alpha_n^-)+1)+(-1)^{i}v_i\right) +\sum_{1\leq i<j\leq N(t)}v_iv_jq_{i}(\alpha_n^-)\Lambda_{j}(\alpha_n^-) \\
		&=\frac{1}{2} \sum_{i=1}^{N(t)}(-1)^{i}v_i\left(1-v_ic_{i}(\alpha_n^-)-|\Lambda_{i}|\right) +\sum_{j=1}^{N(t)}v_j\Lambda_{j}(\alpha_n^-)\sum_{i=1}^{j-1}v_iq_{i}(\alpha_n^-).
	\end{align*}
	By the properties of the Ostrowski representation we get that $v_ic_i(\alpha_n^-)\leq1$, $v_i|\Lambda_i|\leq 1$  and $\sum_{i=1}^{j-1}v_iq_{i}(\alpha_n^-) < q_j(\alpha_n^-)$. Thus, by additionally using that  $N(t)=\calO(\log(t))$ we have 
	\begin{equation}\label{eq:DtAsymptotics}
		D_t(\alpha_n^-)= \frac{1}{2} \sum_{i=1}^{N(t)}(-1)^{i}v_i\left(1-v_ic_{i}(\alpha_n^-)\right)+ \calO(\log(t)).
	\end{equation}
	Inserting \eqref{eq:DtAsymptotics} in the assertion of Theorem~\ref{thm:PqnBehaviour} we obtain
	\begin{equation}
		P_{q_n}(\alpha)=\Theta\left(\frac{1}{a_n}\exp\left(-2\frac{1}{a_n}\sum_{t=1}^{M_n-1}\frac{D_t(\alpha_n^-)}{t(t+1)}\right)\right)=\Theta\left(\frac{1}{a_n}\exp\left(\frac{1}{a_n}Y_n(\alpha)\right)\right),
	\end{equation}
	where $Y_n=\sum_{t=1}^{M_n-1}1/(t(t+1))\sum_{i=1}^{N(t)}(-1)^{i-1}v_i\left(1-v_ic_{i}(\alpha_n^-)\right)$. This verifies also the second statement of Corollary~\ref{cor:PqnBehaviourTypeOfCFC}.
	\end{proof}

	%%%%%%%%%%%%%%%%%%%%%%%%%%%%%%%%%% Proof of Theorem 1.7 %%%%%%%%%%%%%%%%%%%%%%%%%%%%%%%%%%%%%%

	Let us now shift our attention to the proof of Theorem~\ref{thm:ConvBoundedCFC}. We will show that the factors $A_n$, $B_n$ and $C_n$ will converge, given that $\alpha$ has bounded continued fraction coefficients and $\lim_{n\to \infty}\alpha_n^+$ and $\lim_{n\to \infty}\alpha_n^-$ exist.
	We split the proof of Theorem~\ref{thm:ConvBoundedCFC} into three parts, one part for each of the factors $A_n$, $B_n$ and $C_n$. Before starting with the convergence of $A_n$ we introduce some notation that will be helpful later on, especially for the convergence of $B_n$ and $C_n$.
We assume that the following limits exist and set
		\begin{equation}
			\alpha_{\infty}^+:= \lim_{n\to \infty}\alpha_n^+ \text{ and } \alpha_{\infty}^-:=
			\lim_{n\to\infty}\alpha_{n}^-.
		\end{equation}
	Further,
	\begin{align}
			f_n^+:= |\alpha_{n}^+-\alpha_{\infty}^+| \text{ and } f_n^-:= |\alpha_{n}^- -\alpha_{\infty}^-|.
		\end{align}
	By definition we get that $f_n^+,f_n^- \rightarrow 0^+$ for $n\to \infty$ and we also have 
		\begin{equation}\label{eq:alphaInfty}
			\alpha_{n}^+ = \alpha_{\infty}^+ + \calO(f^+_n) \text{ and } \alpha_{n}^- = \alpha_{\infty}^- + \calO(f^-_n).
			\end{equation}
	%%%%%%%%%%%%%%%%%%%%%%%%%%%%   Convergence of A_n   %%%%%%%%%%%%%%%%%%%%%%%%%%%%%%%%%
	
	\begin{proof}[Proof of the convergence of $A_n$:]
		The desired convergence for $A_n$ follows immediately from the assumptions of Theorem~\ref{thm:ConvBoundedCFC} together with Lemma~\ref{lem:StructurePqn}. We have that
		\begin{equation}\label{eq:ConvAn}
			\lim_{n\to\infty} A_n = \lim_{n\to\infty}\frac{2\pi}{\alpha_n^++\alpha_n^-} \left(1+\calO(\Lambda_n^2)\right) = \frac{2\pi}{\alpha_{\infty}^+ + \alpha_{\infty}^-}.
		\end{equation}
		
		\end{proof}
		
		%%%%%%%%%%%%%%%%%%%%%%% Convergence of Cn  %%%%%%%%%%%%%%%%%%%%%%%%%%%%%%%%%

		We need a little bit more effort to prove the convergence of $C_n$.
		
		\begin{proof}[Proof of the convergence of $C_n$:]
		Recall that by Lemma~\ref{lem:StructurePqn} we have that
		\begin{equation}\label{eq:RecallCn}
		C_n= \prod_{t=1}^{\kappa_n} \left( 1- \frac{1}{4\left( t/ c_{n}-\xi_{n t} \right)^2}\right)+\calO\left(\kappa_n^{-1}\right).
		\end{equation}
		As stated in Lemma~\ref{lem:StructurePqn} and also pointed out in Remark~\ref{rem:differentChoiceKappa} equation \eqref{eq:RecallCn} also holds for the choice 
		$$\kappa_n=\floor{\min\{1/(f^-_n)^{1/2}, q_n^{1/2}\}}.$$
		Observe that by this choice we have for $t\in\{1,\ldots,\kappa_n\}$ that $tf^-_n\leq \kappa_nf^-_n \leq \left(f^-_n\right)^{1/2}\rightarrow0$ for $n \to \infty$.	Now together with \eqref{eq:alphaInfty} it follows that for sufficiently large $n$ we obtain
		\begin{align*}
			4\left( t/ c_{n}-\xi_{n t} \right)^2& = 4\left(t(\alpha_n^+ + \alpha_n^-) - \{t\alpha_{n}^-\} +1/2\right)^2 \\
			&= 4\left(t(\alpha_\infty^++\alpha_{\infty}^-)-\{\alpha_{\infty}^-t\}+1/2\right)^2 + t^2\calO(f^+_n + f^-_n)\\
			&= u_{\infty}^2(t)+t^2\calO(f^+_n + f^-_n),
		\end{align*}
		where we have set 
		\begin{equation}
		u_{\infty}(t):=2\left(t(\alpha_\infty^++\alpha_{\infty}^-)-\{t\alpha_{\infty}^-\}+1/2\right).
		\end{equation}
		Further, it follows that
\begin{align}\label{eq:CnConvergenceHelp}
	\prod_{t=1}^{\kappa_n} \left( 1- \frac{1}{4\left( t/ c_{n}-\xi_{n t} \right)^2}\right) %&= \prod_{t=1}^{\kappa_n}\left(1-\frac{1}{u_{\infty}(t)^2+t^2\calO(f^+_n + f^-_n)}\right) \nonumber\\
	&= \prod_{t=1}^{\kappa_n}\left(1-\frac{1}{u_{\infty}(t)^2}\frac{1}{1+t^2/u_{\infty}(t)^2\calO(f^+_n + f^-_n)}\right)\nonumber\\
	&=\prod_{t=1}^{\kappa_n}\left(1-\frac{1}{u_{\infty}(t)^2}-\frac{\calO(f^+_n + f^-_n)}{u_{\infty}(t)^2}\right)\nonumber\\
	&=\prod_{t=1}^{\kappa_n}\left(1-\frac{1}{u_{\infty}(t)^2}\right)\prod_{t=1}^{\kappa_n}\left(1-\frac{\calO(f^+_n + f^-_n)}{u_{\infty}(t)^2-1}\right).
\end{align}
Let us analyze the convergence of the two products in \eqref{eq:CnConvergenceHelp} separately. First of all
it is clear that $\sum_{t=1}^{\infty}1/u_{\infty}(t)^2<\sum_{t=1}^{\infty}1/(u_{\infty}(t)^2-1)<\infty$ by comparison with $\sum_{t=1}^{\infty}1/t^2$. The convergence of the product $\prod_{t=1}^{\kappa_n}\left(1-1/u_{\infty}(t)^2\right)$ now follows by the fact that this is a decreasing sequence in $n$ (note: $u_\infty(t)>2a_n-1\geq1$) and the product is bounded from below by 
\begin{equation}\label{eq:ConvProd_ut}
	\prod_{t=1}^{\kappa_n}\left(1-\frac{1}{u_{\infty}(t)^2}\right) \geq 1- \sum_{t=1}^{\infty}\frac{1}{u_{\infty}(t)^2}\geq 1- \frac{1}{u_{\infty}(1)^2}- \sum_{t=2}^{\infty}\frac{1}{(2t-1)^2}>\frac{3}{5}.
\end{equation}
Note that we have used Lemma~\ref{lem:ProductSumLemma} and the fact that $\alpha_\infty^+ > a_n \geq 1$. %Moreover, it follows by \eqref{eq:ConvProd_ut} that $\prod_{t=1}^{\infty}(1-1/u_{\infty}(t)^2)>0$.
For the second product observe that $\calO(f^+_n + f^-_n)\sum_{t=1}^{\infty}1/(u_{\infty}(t)^2-1) < 1$ for sufficiently large $n$ and we get again by Lemma~\ref{lem:ProductSumLemma} that
\begin{align}\label{eq:ConvFirstProd}
	1-S\leq \prod_{t=1}^{\kappa_n}\left(1-\frac{\calO(f^+_n + f^-_n)}{u_{\infty}(t)^2-1}\right)\leq\left(1-S\right)^{-1},
\end{align}
where $S:=\calO(f^+_n+f^-_n)\sum_{t=1}^{\infty}1/(u_{\infty}(t)^2-1) \rightarrow 0$ for $n \to \infty$. Summarizing the convergence analysis of the products in \eqref{eq:CnConvergenceHelp} it follows that
\begin{equation}\label{eq:ConvCn}
	\lim_{n\to\infty}C_n= \prod_{t=1}^{\infty}\left(1-\frac{1}{u_{\infty}^2(t)}\right).
\end{equation}
Observe that $L_1:=\lim_{n\to\infty}C_n >0$ by \eqref{eq:ConvProd_ut}. This finishes the proof of the convergence of $C_n$.
\end{proof}

%%%%%%%%%%%%%%%%%%%%%%%%% Convergence of B_n  %%%%%%%%%%%%%%%%%%%%%%%%%%%%%%%%%%%%%

Finally we have to show the convergence of the factor $B_n$. Recall that by Lemma~\ref{lem:StructurePqn} we have that
\begin{align}\label{eq:RecallBnLemma}
	\log(B_n)=& -2 \pi^2\frac{c_n}{q^2_n}\sum_{t=1}^{M_n-1}\frac{D_t(\alpha_n^-)}{\sin(\pi t/q_n)\sin(\pi(t+1)/q_n)} \nonumber\\
	&-2\sum_{t=1}^{\tau_n}\sum_{j=2}^{\tau_n} \frac{1}{j}\left(\frac{c_{n}\xi_{nt}}{t}\right)^j	+ \calO\left(\tau_n^{-1}\right).
\end{align}
As stated in Lemma~\ref{lem:StructurePqn} and also indicated in Remark~\ref{rem:ChoiceOftaun} equation \eqref{eq:RecallBnLemma} also holds for 
\begin{equation}\label{eq:ChoiceTaun}
\tau_n:=\floor{\min\{\log_2\left((f^+_n+f^-_n)^{-1/2}\right), q_n^{1/2} \}}.
\end{equation}
A consequence of this choice is that for $1\leq t,j \leq \tau_n$ it follows that for some fixed $k\in\NN$
\begin{align}
	&t^k f_n^- \leq -\frac{1}{2}\left(\log_2(f^-_n)\right)^kf_n^- \rightarrow 0, \text{ for } n\to \infty; \label{eq:ChoiceTauConsequence1}\\
	& 2^jf_n^- \leq (f_n^+ + f_n^-)^{-1/2}f_n^- \leq \sqrt{f_n^-} \rightarrow 0, \text{ for } n \to \infty.\label{eq:ChoiceTauConsequence2}
\end{align}
These facts will be useful several times in the convergence analysis of $B_n$.
\begin{proof}[Proof of convergence of $B_n$]
We see from \eqref{eq:RecallBnLemma} that it is sufficient for the convergence of $B_n$ to prove convergence of the following two terms:
\begin{equation}\label{eq:twoTerms}
	\pi^2\frac{c_n}{q^2_n}\sum_{t=1}^{M_n-1}\frac{D_t(\alpha_n^-)}{\sin(\pi t/q_n)\sin(\pi(t+1)/q_n)} \text{ and } \sum_{t=1}^{\tau_n}\sum_{j=2}^{\tau_n} \frac{1}{j}\left(\frac{c_{n}\xi_{nt}}{t}\right)^j.
\end{equation}
We start with the double sum term in \eqref{eq:twoTerms}. Let us first analyze the summands more detailed. It follows by \eqref{eq:cn} and \eqref{eq:xint} for sufficiently large $n$ that
\begin{align}\label{eq:SummandsDoubleSum}
	\frac{c_n\xi_{nt}}{t} = \frac{\{t\alpha_n^-\}-1/2}{t(\alpha_n^++\alpha_n^-)} = \frac{\{t\alpha_{\infty}^-\} -1/2}{t(\alpha_{\infty}^++\alpha_{\infty}^-)}\left(1+\calO(f^+_n + f^-_n)\right) + \calO(f_n^{-}).
\end{align}
In order to simplify notation let us set $v_t:=(\{t\alpha_{\infty}^-\} -1/2)/(t(\alpha_{\infty}^++\alpha_{\infty}^-))$ and recall that $v_t<1/(2t)$ and $f_n^+,f_n^- \rightarrow 0$ for $n\to\infty$.
Since we have for $x,y\leq1$
$$(x+y)^j \leq x^j +2^jy$$
it follows for sufficiently large $n$ together with \eqref{eq:SummandsDoubleSum} that
\begin{align}\label{eq:SummandPower}
	\left(\frac{c_n\xi_{nt}}{t}\right)^j	&= v_t^j\left(1+\calO(f^+_n + f^-_n)\right)^j  + \calO(2^jf_n^-)\nonumber\\
	&= v_t^j\left(1 + \calO(2^j(f^+_n + f_n^-))\right) + \calO(\sqrt{f_n^-})\nonumber\\
	&= v_t^j + \calO(\sqrt{f_n^-}).
\end{align}
Note that we used \eqref{eq:ChoiceTauConsequence2} in the last step. It follows by \eqref{eq:SummandPower} together with the choice of $\tau_n$ that
%Observe that it follows again by the choice of $\tau_n$ in \eqref{eq:ChoiceTaun} that $tf^-_n \leq \tau_nf^-_n < (f^-_n)^{2/3} < 1$ for all $t\in \{1,\ldots,\tau_n\}$ if $n$ is sufficiently large.
%We also have for each $t\in\{1,\ldots,\tau_n\}$ that $tf^-_n + f^+_n < 1$ if $n \to \infty$. Thus, it is not hard to check that
%\begin{equation}\label{eq:HelpBigO}
	%\left(1 + \calO(tf^-_n + f^+_n)\right)^j = 1 + \calO\left(2^j(tf^-_n + f^+_n)\right).
%\end{equation}
%In order to simplify notation let us set $v_t:=(\{t\alpha_{\infty}^-\} -1/2)/(t(\alpha_{\infty}^++\alpha_{\infty}^-))$ and recall that $2|v_t| < 1/t\leq1$. Now by combining \eqref{eq:SummandsDoubleSum} and \eqref{eq:HelpBigO} we obtain
\begin{align*}
	\sum_{t=1}^{\tau_n}\sum_{j=2}^{\tau_n} \frac{1}{j}\left(\frac{c_n\xi_{nt}}{t}\right)^j %&= \sum_{t=1}^{\tau_n}\sum_{j=2}^{\tau_n} \frac{1}{j}v_t^j + \sum_{t=1}^{\tau_n}\sum_{j=2}^{\tau_n} \calO(\sqrt{f_n^-}) \\
	&=\sum_{t=1}^{\tau_n}\sum_{j=2}^{\tau_n} \frac{1}{j}v_t^j + \calO\left(\tau_n^2\sqrt{f_n^-}\right) =\sum_{t=1}^{\tau_n}\sum_{j=2}^{\tau_n} \frac{1}{j}v_t^j+\calO\left(\frac{\sqrt{f^-_n}}{\log_2((f^-_n)^{-1/2})}\right).
\end{align*}
Furthermore we see that 
$$ \sum_{t=1}^{\tau_n}\sum_{j=2}^{\tau_n}\left| \frac{1}{j}v_t^j\right| \leq\frac{1}{2}\sum_{t=1}^{\infty}\sum_{j=2}^{\infty} \frac{1}{(2t)^j} = \frac{1}{2}\sum_{t=1}^{\infty}\frac{1}{2t(2t-1)} <\infty.$$
Thus $\sum_{t=1}^{\tau_n}\sum_{j=2}^{\tau_n} \frac{1}{j}v_t^j$ is an absolutely convergent series and therefore there exists a constant $L_2$ with 
\begin{equation}\label{eq:ConvSecondTerm}
\lim_{n\to\infty}\sum_{t=1}^{\tau_n}\sum_{j=2}^{\tau_n} \frac{1}{j}\left(\frac{\{t\alpha_n^-\}-1/2}{t(\alpha_n^++\alpha_n^-)}\right)^j=\lim_{n\to\infty}\sum_{t=1}^{\tau_n}\sum_{j=2}^{\tau_n} \frac{1}{j}v_t^j=L_2.
\end{equation}
Let us shift our attention to the second term in \eqref{eq:twoTerms}. Making use of $\sin(x)=x(1+\calO(x^2))$ we obtain for sufficiently large $n$ that
\begin{align}
	\pi^2\frac{c_n}{q^2_n}&\sum_{t=1}^{M_n-1}\frac{D_t(\alpha_n^-)}{\sin(\pi t/q_n)\sin(\pi(t+1)/q_n)} = c_n\sum_{t=1}^{M_n-1}\frac{D_t(\alpha_n^-)}{t(t+1)}\left(1+\calO(t^2q_n^{-2})\right)\nonumber\\
	&=c_n\sum_{t=1}^{\tau_n}\frac{D_t(\alpha_n^-)}{t(t+1)} + c_n\sum_{t=\tau_n+1}^{M_n-1}\frac{D_t(\alpha_n^-)}{t(t+1)} + \calO\left(q_n^{-2}\sum_{t=1}^{M_n}D_t(\alpha_n^-)\right) \nonumber\\
	&= c_n\sum_{t=1}^{\tau_n}\frac{D_t(\alpha_n^-)}{t(t+1)} + \calO\left(\frac{\log(\tau_n)}{\tau_n}+\frac{\log(q_n!)}{q_n^2}\right),
\end{align}
where we have used in the last step that $D_t(\alpha_n^-)=\calO(\log(t))$ for each $t\in\{1,\ldots,M_n-1\}$ (see Section~\ref{sec:Prelim}, equation \eqref{eq:Dtbcfc}). Additionally by \eqref{eq:ChoiceTauConsequence1} we have for sufficiently large $n$ that
\begin{align}\label{eq:SumDt}
	\sum_{t=1}^{\tau_n}\frac{D_t(\alpha_n^-)}{t(t+1)}&=\sum_{t=1}^{\tau_n}\frac{1}{t(t+1)}\sum_{s=1}^t\left(\{s\alpha_n^-\}-\frac{1}{2}\right)\nonumber\\
	&= \sum_{t=1}^{\tau_n}\frac{1}{t(t+1)}\sum_{s=1}^t\left( \{s\alpha_{\infty}^-\}-\frac{1}{2} + s\calO\left(f^-_n\right)\right)\nonumber\\
	&=\sum_{t=1}^{\tau_n}\left[\frac{1}{t(t+1)}\sum_{s=1}^t\left(\{s\alpha_{\infty}^-\}-\frac{1}{2}\right)\right] + \calO\left(\tau_nf^-_n\right).
\end{align}
%This means that  
%\begin{equation}\label{eq:SumDt}
	%\sum_{t=1}^{\tau_n}\frac{D_t(\alpha_n^-)}{t(t+1)} = \sum_{t=1}^{\tau_n}\left[\frac{1}{t(t+1)}\sum_{s=1}^t\left(\{t\alpha_{\infty}^-\} - \frac{1}{2}\right)\right]\left(1 + \calO\left(\sqrt{f^-_n}\right)\right).
%\end{equation}
Observe that by the choice of $\tau_n$ in \eqref{eq:ChoiceTaun} we clearly have that $\tau_nf_n^- \rightarrow 0$ for $n\to\infty$. Additionally note that by \eqref{eq:Dtbcfc} it follows that for some constant $K>0$ independent of $n$ we have
$$\left|\sum_{t=1}^{\tau_n}\frac{1}{t(t+1)}\sum_{s=1}^t\left(\{s\alpha_{\infty}^-\} - \frac{1}{2}\right)\right| \leq K\sum_{t=1}^{\infty}\frac{\log(t)}{t^2}<\infty.$$
Therefore, $\sum_{t=1}^{\tau_n}1/(t(t+1))\sum_{s=1}^t\left(\{s\alpha_{\infty}^-\} - 1/2\right)$ is an absolutely convergent series. Together with \eqref{eq:SumDt} this means that there exists a real number $L_3$ such that
\begin{equation}\label{eq:ConvFirstTerm}
	\lim_{n\to\infty}\sum_{t=1}^{\tau_n}\frac{D_t(\alpha_n^-)}{t(t+1)} = L_3.
\end{equation}
Finally, by combining \eqref{eq:RecallBnLemma} with \eqref{eq:ConvSecondTerm} and \eqref{eq:ConvFirstTerm} we get that
\begin{align}\label{eq:ConvBn}
	\lim_{n\to\infty} B_n &= \lim_{n\to\infty}\exp\left( -2c_{n}\sum_{t=1}^{\tau_n}\frac{D_t(\alpha_n^-)}{t(t+1)} -2\sum_{t=1}^{\tau_n}\sum_{j=2}^{\tau_n} \frac{1}{j}\left(\frac{c_{n}\xi_{nt}}{t}\right)^j
		 \right) \nonumber\\
		&= \exp\left(-\frac{2L_3}{\alpha_{\infty}^++\alpha_{\infty}^-}-2L_2\right)>0.
\end{align}
This establishes the convergence of $B_n$.
\end{proof}
Finally, the statement of Theorem~\ref{thm:ConvBoundedCFC} is now a direct consequence of the convergence of the factors $A_n$, $B_n$ and $C_n$. By \eqref{eq:ConvAn}, \eqref{eq:ConvCn} and \eqref{eq:ConvBn} it follows that
\begin{align*}
	\lim_{n\to\infty}P_{q_n}(\alpha) &= \lim_{n\to\infty}A_n \lim_{n\to\infty}B_n\lim_{n\to\infty}C_n \nonumber \\
	&= \frac{2\pi}{\alpha_{\infty}^+ + \alpha_{\infty}^-} \exp\left(-\frac{2L_2}{\alpha_{\infty}^++\alpha_{\infty}^-}-2L_3\right)L_1>0.
\end{align*}

Last but not least we have to prove Corollary~\ref{cor:behaviourEuler}. 
\begin{proof}[Proof fo Corollary~\ref{cor:behaviourEuler}]
Recall that $\alpha=\mathrm{e}=[2;\overline{1,2n,1}]_{n=1}^{\infty}$. Let $n_i^{(k)}=n_i=3i+k$ for $k\in\{0,1,2\}$ and recall that by \eqref{eq:defAlphPlusAlphMinus} we have
$$\alpha^-_{n_i}=[0;a_{n_i-1},a_{n_i-2},\ldots,a_1]$$
and $a_{3i}=1$ for $i\geq1$, $a_{3i+1}=1$ and $a_{3i+2}=2(i+1)$ for $i\geq 0$.
Now consider the Ostrowski representation in base $\alpha^-_{n_i}$ for $t \in \{1,\ldots,q_{n}(\alpha)-1\}=\{1,\ldots,q_n(\alpha_{n_i})-1\}$ i.e. $t=\sum_{j=1}^{N(t)}v_j(t)q_j(\alpha_{n_i}^-)$. (As already mentioned in Section~\ref{sec:Prelim} this representation is well defined.)	By Theorem~\ref{thm:PqnBehaviour} we get that 
	%\begin{equation}\label{eq:structurePqni}
		%P_{q_{n_i}}(\mathrm{e}) =\Theta\left( c_{n_i}\exp\left(c_{n_i} \sum_{t=1}^{M_{n_i}-1} \frac{1}{t(t+1)} \sum_{j=1}^{N(t)}(-1)^{j-1}z_j(t)\delta_j(t)\right)\right),
		%\end{equation}
%	where we set $\delta_j(t) = 1- z_j(t)c_j(\alpha_{n_i}^-).$ Note that we have
	%\begin{equation}
		%\delta_j(t) \leq 1-\frac{1}{b_j^{(i)} +2} < 1.
	%\end{equation}
	\begin{equation}\label{eq:structurePqni}
		P_{q_{n_i}}(\mathrm{e}) =\Theta\left( c_{n_i}\exp\left(c_{n_i}\sum_{t=1}^{M_{n_i}-1}\frac{y_{n_i}(t)}{t(t+1)} \right)\right),
		\end{equation}
		where we have set
		\begin{align*}
		y_{n_i}(t):= \sum_{j=1}^{N(t)}(-1)^{j-1}v_j(t)(1-v_j(t)c_j(\alpha_{n_i}^-)).
		\end{align*}
Let $k=0$ i.e. $n_i=3i$, then $\alpha_{n_i}^-=[0;2i,1,1,2(i-1),1,1,\ldots,1]$. For sufficiently large $i$ we can split up the sum over $t$ appearing in \eqref{eq:structurePqni} in the following way
\begin{align*}
	\sum_{t=1}^{M_{n_i}-1}\frac{y_{n_i}(t)}{t(t+1)} &= \sum_{t=1}^{q_2(\alpha_{n_i}^-)-1}\frac{y_{n_i}(t)}{t(t+1)}+\sum_{t=q_2(\alpha_{n_i}^-)}^{q_5(\alpha_{n_i}^-)-1}\frac{y_{n_i}(t)}{t(t+1)}+\sum_{t=q_5(\alpha_{n_i}^-)}^{M_{n_i}-1}\frac{y_{n_i}(t)}{t(t+1)}\\
	&=:S_1 + S_2 +S_3.
\end{align*}
We compute that 
\begin{align*}
	q_2(\alpha_{n_i}^-)=2i; ~ q_3(\alpha_{n_i}^-) = 2i+1;~
	q_4(\alpha_{n_i}^-)=4i+1;~ q_5(\alpha_{n_i}^-)=8i^2-4i-1.
\end{align*}
	Observe further that the Ostrowski representation of $t$ in base $\alpha_{n_i}$ collapses to $t=v_1(t)$ for $t\in\{1,\ldots,q_2(\alpha_{n_i}^-)-1\}$.  Using these observations it follows that
	\begin{align}\label{eq:boundS1}
		S_1 &= \sum_{t=1}^{q_2(\alpha_{n_i}^-)-1}\frac{1}{t(t+1)}\sum_{j=1}^{N(t)}(-1)^{j-1}v_j(t)(1-v_j(t)c_j(\alpha_{n_i}^-))\nonumber\\
		&=\sum_{t=1}^{2i-1}\frac{1}{(t+1)}-c_1(\alpha_{n_i}^-)\sum_{t=1}^{2i-1}\frac{t}{(t+1)}\nonumber\\
		&\geq \log(2i)-1 - \frac{2i-1}{a_{n_i-1}+2} \geq \log(2i)-2, 
	\end{align}
	where we applied $c_1(\alpha_{n_i})\leq 1/a_{n_i-1}=1/(2i).$ For the second sum we obtain
	\begin{align}\label{eq:boundS2}
		S_2 &= \sum_{t=2i}^{8i^2-4i-2}\frac{1}{t(t+1)}\sum_{j=1}^{4}(-1)^{j-1}v_j(t)(1-v_j(t)c_j(\alpha_{n_i}^-)) \nonumber\\
		&\geq -\sum_{t=2i}^{\infty}\frac{1}{t^2}\sum_{j=1}^{4}a_{n_i-j} = -4i\sum_{t=2i}^{\infty}\frac{1}{t^2}  \geq - 4,
	\end{align}
	where we have applied the bounds $1-v_j(t)c_j(\alpha_{n_i})<1$ and $v_j(t)\leq a_{n_i-j}$ which are valid for all $t$. We can apply similar arguments for the third sum $S_3$ and get
	\begin{align}\label{eq:boundS3}
		S_3 &= \sum_{t=8i^2-4i-1}^{M_{n_i}-1}\frac{1}{t(t+1)}\sum_{j=1}^{N(t)}(-1)^{j-1}v_j(t)(1-v_j(t)c_j(\alpha_{n_i}^-))\nonumber\\
		&\geq - \sum_{t=7i^2+1}^{\infty}\frac{1}{t^2}\sum_{j=1}^{n_i-1}a_{n_i-j} = -10i^2\sum_{t=7i^2+1}^{\infty}\frac{1}{t^2} \geq -\frac{10}{7},
	\end{align}
	where we make use of the rough bound $\sum_{j=1}^{n_i-1}a_j \leq 10i^2$ which holds for sufficiently large $i$. Combining \eqref{eq:structurePqni}, \eqref{eq:boundS1}, \eqref{eq:boundS2} and \eqref{eq:boundS3} with the fact that 
	$$ \frac{1}{3} \leq \frac{1}{a_{3i}+2} < c_{3i}<\frac{1}{a_{3i}} \leq 1$$
	we get that
\begin{equation*}
	P_{q_{3i}}(\mathrm{e}) \geq C_1 \sqrt[3]{i}
\end{equation*}
for some constant $C_1>0$ independent of $i$. This proves that $P_{q_{3i}}(\mathrm{e})\rightarrow \infty$ for $i\to\infty$ and finishes the case $k=0$.\\
If we consider the case $k=1$ (i.e. $n_i=3i+1$ and $\alpha_{n_i}=[0;1,2i,1,1,2(i-1),\ldots,1]$) the crucial observation in this case is that $q_2(\alpha_{n_i}^-)=1$ and $a_{n_i-2}=2i$. Thus for all $t\in\{1,\ldots,q_3(\alpha_{n_i})-1\}$ we have that the Ostrowski expansion of $t$ in base $\alpha_{n_i}^-$ is $t=v_2(t)$. Additionally $v_1(t)=0$ for all $t$. Performing a similar splitting of the sum in \eqref{eq:structurePqni} as in the case $k=0$ we get that
\begin{align*}
	\sum_{t=1}^{M_{n_i}-1}\frac{y_{n_i}(t)}{t(t+1)} &= \sum_{t=1}^{q_3(\alpha_{n_i}^-)-1}\frac{y_{n_i}(t)}{t(t+1)}+\sum_{t=q_3(\alpha_{n_i}^-)}^{q_6(\alpha_{n_i}^-)-1}\frac{y_{n_i}(t)}{t(t+1)}+\sum_{t=q_6(\alpha_{n_i}^-)}^{M_{n_i}-1}\frac{y_{n_i}(t)}{t(t+1)}\\
	&=:\overline{S}_1 + \overline{S}_2 + \overline{S}_3. 
\end{align*}
Following the same arguments as in the case $k=0$ it follows that for sufficiently large $i$ we get
\begin{align*}
	&\overline{S}_1 = -\sum_{t=1}^{2i}\frac{1}{t+1} + c_2(\alpha_{n_i}^-)\sum_{t=1}^{2i}\frac{t}{t+1} \leq -\log(2i+1) +2;\\
	&\overline{S}_2 \leq 4i \sum_{t=2i+1}^{\infty}\frac{1}{t^2} \leq 2 ~\text{ and }~
	\overline{S}_3  \leq 10i^2 \sum_{t=7i^2+1}^{\infty}\frac{1}{t^2} \leq \frac{10}{7}.
\end{align*}
Making use of $1/3\leq c_{3i+1} \leq 1$ we get by \eqref{eq:structurePqni} that
\begin{equation}
	P_{q_{3i+1}}(\mathrm{e}) \leq C_2 \frac{1}{\sqrt[3]{2i+1}} 
\end{equation}
for some constant $C_2>0$ independent of $i$. It follows that $P_{q_{3i+1}}(\mathrm{e})\rightarrow 0$ for $i \to \infty$.\\
For the last case $k=2$ (i.e. $n_i=3i+2$ and $\alpha_{n_i}=[0;1,1,2i,1,1,2(i-1),\ldots,1]$) we observe that $q_2(\alpha_{n_i}^-)=1$, $q_3(\alpha_{n_i}^-)=2$ and $a_{n_i-3}=2i$. In this case the digits of the Ostrowski expansion in base $\alpha_{n_i}^-$ satisfy $v_1(t)=0$ for all $t$, $v_2(t)=1$ if $t$ is odd and $v_2(t)=0$ if $t$ is even. Further, the following property holds for $t\in\{1,\ldots,q_4(\alpha_{n_i})-1\}$.
$$t= 
\begin{cases}
	2v_3(t) & t \text{ even,}\\
	1+2v_3(t) & t \text{ odd.}
\end{cases}$$
Splitting the sum in \eqref{eq:structurePqni} gives
\begin{align*}
	\sum_{t=1}^{M_{n_i}-1}\frac{y_{n_i}(t)}{t(t+1)} &=\sum_{t=1}^{q_4(\alpha_{n_i}^-)-1}\frac{y_{n_i}(t)}{t(t+1)}+\sum_{t=q_4(\alpha_{n_i}^-)}^{q_7(\alpha_{n_i}^-)-1}\frac{y_{n_i}(t)}{t(t+1)}+\sum_{t=q_7(\alpha_{n_i}^-)}^{M_{n_i}-1}\frac{y_{n_i}(t)}{t(t+1)}\\
	&=:\tilde{S}_1 + \tilde{S}_2 +\tilde{S}_3.
\end{align*}
Following the same steps as in the cases $k=0$ and $k=1$ we get that
\begin{align*}
	\tilde{S}_1 &\leq \frac{1}{2} \sum_{t=2}^{4i}\frac{1}{t} 		\leq \frac{1}{2}\log(4i+1);\\
	\tilde{S}_2 &\leq (4i+1) \sum_{t=4i+1}^{\infty}\frac{1}{t^2} \leq \frac{5}{4} \text{ and }	\tilde{S}_3 \leq 10i^2 \sum_{t=15i^2+1}^{\infty}\frac{1}{t^2} \leq \frac{1}{3}.
\end{align*}
Additionally we have $1/(a_{3i+2}+2)\leq c_{3i+2}(\alpha) \leq 1/a_{3i+2}$ and $a_{3i+2}=2(i+1)$. Hence, by \eqref{eq:structurePqni} we get
\begin{equation*}
	P_{q_{3i+2}}(\mathrm{e}) \leq C_3 \frac{1}{i+1}\exp\left(\frac{\log(4i+1)}{4(i+1)}+1\right)
\end{equation*}
for some constant $C_3>0$ independent of $i$, meaning that $P_{q_{3i+2}}(\mathrm{e}) \rightarrow 0$ for $i\to\infty$. This finishes the proof of Corollary~\ref{cor:behaviourEuler}.
\end{proof}

\noindent
\textbf{Acknowledgments}\\
The author would like to thank Sigrid Grepstad for helpful advices and fruitful discussions about several topics related to this article.

	%%%%%%%%%%%%%%%%%%%%%%%%%%%%%%%%%%%%%%%%%%%%%%%%%%%%%%%%%%%%%%%%%%%%%%%%%%%%%%%%%%%%%%%%%%%%%%
	% BIBLIOGRAPHY
	%%%%%%%%%%%%%%%%%%%%%%%%%%%%%%%%%%%%%%%%%%%%%%%%%%%%%%%%%%%%%%%%%%%%%%%%%%%%%%%%%%%%%%%%%%%%%%
	\bibliography{mybibl}
	\bibliographystyle{plain}

\end{document}